\PassOptionsToPackage{table}{xcolor}

\documentclass[smallextended]{svjour3}
\smartqed

\usepackage{url}

\usepackage{amsfonts}
\usepackage{amssymb} 
\usepackage{amsmath}
\usepackage{comment}
\usepackage[utf8]{inputenc}

\usepackage{xcolor}

\usepackage{theorem}

\newtheorem{rem}{Remark}
\newtheorem{prop}{Proposition}
\newtheorem{thm}{Theorem}
\newtheorem{lem}{Lemma}
\newtheorem{cor}{Corollary}

\usepackage{adjustbox}

\usepackage{ragged2e}
\usepackage{array}
\usepackage{adjustbox}
\usepackage[table]{xcolor}

\usepackage{fancyhdr}

\usepackage{array,multirow,graphicx}

\usepackage{supertabular}

\usepackage{adjustbox}

\usepackage{epstopdf}\usepackage{epsfig}\usepackage[utf8]{inputenc}\usepackage{longtable,booktabs}\usepackage{lscape}\usepackage{hhline}\usepackage{caption}\usepackage{subfig}\usepackage{blindtext}\usepackage{lmodern}\usepackage{textcomp}\usepackage{chngcntr}\usepackage{booktabs}

\fancypagestyle{plain}{\fancyhf{}\fancyfoot[C]{\bfseries \thepage}}

\newcolumntype{?}{!{\vrule width 1pt}}

\usepackage[table]{xcolor}

\arrayrulecolor[HTML]{000000}

\def\Box{{\hbox{\raisebox{0.0em}{\rlap{$\sqcap$}}\kern0em%
			\raisebox{-0.0emp}{$\sqcup$}}} } 

\def\bfi#1{\textbf{#1}} 
\def\att{					
	\marginpar[ \hspace*{\fill} \raisebox{-0.2em}{\rule{2mm}{1.2em}} ]
	{\raisebox{-0.2em}{\rule{2mm}{1.2em}} }
}
\def\at#1{[*** \att #1 ***]}  

\def\eeq{\end{equation}}
\def\lbeq#1{\begin{equation} \label{#1}}
\def\bary{\begin{array}}
	\def\eary{\end{array}}
\def\gzit#1{{\rm (\ref{#1})}} 			

\def\D{\displaystyle}				
\def\ol{\overline}

\def\wt{\widetilde}

\def\spc{~~~}

\def\scal{{\fns{sc}}}
\def\modify{{\fns{md}}}

\def\Areg{A$_{\fct{reg}}$}

\def\eps{\varepsilon}

\def\Rz{\mathbb{R}}

\def\cc {c_E}

\def\dist{\fct{dist}}

\def\fct#1{\mathop{\rm #1}}	                
\def\fns#1{{\mbox{\rm \scriptsize#1}}} 		

\def\init{\fns{init}}

\def\angle{\fct{a}}

\def\spc{\hspace*{0.5cm}} 			
\def\D{\displaystyle}				
\def\ol{\overline}

\def\ealg{\end{alg}}

\usepackage{algorithm}
\usepackage{algpseudocode}

\makeatletter
\newcommand{\algrule}[1][.2pt]{\par\vskip.5\baselineskip\hrule height #1\par\vskip.5\baselineskip}
\makeatother

\def\bfi#1{{\bf{#1}}}

\def\Asol{A$_{\fct{sol}}$}
\def\Alip{A$_{\fct{lip}}$}
\def\Aeb{A$_{\fct{eb}}$}

\usepackage{tikz}
\usepackage{pgfplots}

\usepackage[colorlinks=true, allcolors=blue]{hyperref}
\usepackage{orcidlink}

\begin{document}
	\title{Subspace methods for min-max problems}

\author{Morteza Kimiaei$^\dagger$\thanks{$\dagger$ MK acknowledges financial support from the Austrian Science Foundation under https://doi.org/10.55776/PAT2747625} \and {Shima Shabani$^\ddagger$ \and Michael Breu\ss$^\ddagger$}\thanks{$\ddagger$ SS $\&$ MB acknowledge funding from the Federal Ministry of Research, Technology and Space (BMFTR), Germany, through the project “Digital GreenTech (DGT) – Environmental Technology Meets Robotics” (grant number 68542).} }

\institute{
M. Kimiaei (corresponding author)\at
Fakult\"at f\"ur Mathematik, Universit\"at Wien, Oskar-Morgenstern-Platz 1,  A-1090, Wien, Austria\\
\email{morteza.kimiaei@univie.ac.at} 
\and 
S. Shabani and M. Breu\ss \at 
Institute for Mathematics, Brandenburg University of Technology, Platz der Deutschen Einheit 1, 03046 Cottbus, Germany \\
\email{shima.shabani, breuss@b-tu.de}
}

\maketitle

\begin{abstract}
\begin{sloppypar}
This paper introduces four groups of subspace methods for nonlinear monotone equations, with applications to large-scale machine learning problems. The methods use Jacobian-free subspace ({\tt JFS}) directions of conjugate-gradient type, combined with either fixed step sizes or variable step sizes generated by the projected method of Solodov and Svaiter. To ensure convergence independently of the specific algebraic form of the subspace directions, we impose an angle condition together with an explicit scaling rule controlling the effective search directions. Under monotonicity and Lipschitz continuity of the operator, we establish global convergence for both the line-search and fixed-step frameworks, as well as a best-iterate residual rate $O(\ell^{-1/2})$. Under a local error bound, the distance to the solution set satisfies the sharper decay $o(\ell^{-1/2})$. If the operator is continuously differentiable and its Jacobian is nonsingular at a solution, the required local error bound and local isolation follow, yielding $R$-linear local convergence. The residual sequence then converges geometrically and hence satisfies the last-iterate rate $o(\ell^{-1})$, without strong monotonicity. We also derive iteration and residual-evaluation complexity bounds: $O(\varepsilon^{-2})$ for the baseline best-iterate guarantee and $O(\log(\varepsilon^{-1}))$ in the local linear regime, together with a uniform bound on backtracking residual evaluations. Under additional asymptotic assumptions, the proposed {\tt JFS} directions and several classical update directions admit related optimistic gradient descent--ascent (\texttt{OGDA})-type residual--memory representations. Numerical experiments illustrate the robustness and efficiency of the methods on representative min--max problems.

\end{sloppypar}
\end{abstract}


\vspace{0.2cm} {\em 2000 AMS Subject Classification: 47J20, 90C56}.

\begin{sloppypar}
\section{Introduction}\label{sec:into}
The theory of monotone operators was originally developed in the 1960s in the context of functional analysis and partial differential equations. 
It was soon recognized that this framework is also fundamental for convex optimization. 
In the 1970s, iterative algorithms based on monotone operator theory---such as proximal-point and fixed-point methods---were introduced for solving convex minimization problems; see \cite{ryu2016primer} and the references therein. 
Since then, the field has grown substantially and remains highly active, with applications in compressed sensing \cite{HP,XZ}, power engineering systems \cite{WW}, and optimization models arising in machine learning \cite{Jin2021}. Considering modern machine learning, one of the primary catalysts for studying
nonlinear monotone equations (NME) is the structural transition from simple minimization to structured min--max problems that
can be formulated as
\begin{equation*}
        \min_{u \in \mathbb{R}^{n_u}} \max_{v \in \mathbb{R}^{n_v}} \Phi(u, v),
\end{equation*}
where $\Phi : \mathbb{R}^{n_u} \times \mathbb{R}^{n_v} \to \mathbb{R}$ represents the objective function. Under standard convex--concave and differentiability assumptions on $\Phi$, the first-order optimality conditions lead to a monotone operator formulation \cite{Bo2023}. 
Specifically, defining $x=(u,v)\in\mathbb{R}^n$, where
$n=n_u+n_v$, and
\begin{equation*}
    E(x) = 
    \begin{bmatrix} 
        \nabla_u \Phi(u, v) \\ 
        -\nabla_v \Phi(u, v) 
    \end{bmatrix},
\end{equation*}
the first-order optimality conditions for the unconstrained min--max
problem reduce to the nonlinear monotone equation
\begin{equation}\label{e.ME}
E(x) = 0, \qquad x \in \mathbb{R}^n.
\end{equation}
The present paper focuses on the equation formulation \eqref{e.ME}. Under the convex--concave structure, the operator $E$ is monotone on $\Rz^n$, i.e.,
\begin{equation}\label{e.monotone}
\big(E(x_1) - E(x_2)\big)^\top(x_1 - x_2) \geq 0, \qquad \forall x_1, x_2 \in \Rz^n.
\end{equation}
The following three application classes are considered again in the numerical experiments; see Section~\ref{s.num}:
\begin{itemize}
   \item {\bfi{Multi-Agent Reinforcement Learning (MARL)} \cite{jiao2024minimax,omid}{\bfi:}} 
In competitive environments, agents seek a Nash equilibrium. 
For a two-player zero-sum Markov game with discount factor $\gamma \in (0,1)$ and reward sequence $\{r_t\}_{t \ge 0}$, using unconstrained parameterizations $\pi_1(\theta_1)$ and $\pi_2(\theta_2)$ of the players' policies, the value function can be framed as
\begin{equation*}
    \min_{\theta_1 \in \mathbb{R}^{n_1}} \max_{\theta_2 \in \mathbb{R}^{n_2}} 
    \Phi(\theta_1, \theta_2) 
    = 
    \mathbb{E}_{\pi_1(\theta_1), \pi_2(\theta_2)} 
    \left[ 
        \sum_{t=0}^{\infty} \gamma^t r_t 
    \right],
\end{equation*}
where $\mathbb{E}_{\pi_1(\theta_1),\pi_2(\theta_2)}$ denotes expectation with respect to the probability measure induced by the initial state distribution, the transition kernel of the Markov game, and the joint policy $(\pi_1(\theta_1),\pi_2(\theta_2))$. The transition kernel specifies the probability distribution
of the next state given the current state and the actions selected by the two
players. The corresponding first-order stationarity system can be written as an NME in the joint parameter vector $x=(\theta_1,\theta_2)\in\mathbb{R}^{n_1+n_2}$.

\item {\bfi{{Robust Adversarial Learning (RAL)}} \cite{madr,robey2021adversarial}{\bfi:}} 
To ensure robustness against adversarial perturbations, models are trained by solving
\begin{equation*}
    \min_{\theta \in \mathbb{R}^{n_\theta}} 
    \sum_{i=1}^n 
    \max_{\delta_i \in \mathbb{R}^{n_\delta}} 
    \Phi(\theta, \delta_i),
\end{equation*}
where $\theta \in \mathbb{R}^{n_\theta}$ denotes the model parameter vector and
$\delta_i \in \mathbb{R}^{n_\delta}$ represents an adversarial perturbation applied to the $i$th sample. 
The corresponding first-order stationarity conditions define a high-dimensional NME whose operator is typically monotone but not strongly monotone.
\item {\bfi{ {Generative Adversarial Networks (GANs)}} \cite{bohm,goodfellow2020generative}{\bfi:}} 
Let $G_{\theta_G}$ and $D_{\theta_D}$ denote the generator and discriminator
parameterized by $\theta_G \in \mathbb{R}^{n_G}$ and $\theta_D \in \mathbb{R}^{n_D}$, respectively. Let $p_{\text{data}}$ denote a probability distribution on the data space,
$\xi$ a data sample drawn from $p_{\text{data}}$, and $p_z$ a probability
distribution on the latent space. The GAN objective can be written as the min--max problem
\begin{equation*}
    \min_{\theta_G \in \mathbb{R}^{n_G}} 
    \max_{\theta_D \in \mathbb{R}^{n_D}}
    \Phi(\theta_G,\theta_D),
\end{equation*}
where
\[
\Phi(\theta_G,\theta_D)
=
\mathbb{E}_{\xi \sim p_{\text{data}}}[\log D_{\theta_D}(\xi)]
+
\mathbb{E}_{z \sim p_z}[\log(1 - D_{\theta_D}(G_{\theta_G}(z)))].
\]
Here, $\mathbb{E}_{\xi \sim p_{\text{data}}}[\cdot]$ denotes expectation with respect to 
the probability measure $p_{\text{data}}$, and 
$\mathbb{E}_{z \sim p_z}[\cdot]$ denotes expectation with respect to the probability 
measure $p_z$. The associated first-order stationarity system
can be written as an NME in the joint parameter vector
$x=(\theta_G,\theta_D)\in\mathbb{R}^{n_G+n_D}$. For convex--concave GAN-type models, the induced operator is monotone but need
not be strongly monotone. More general nonconvex--nonconcave GAN models fall
outside the monotone setting considered in this paper and may exhibit
oscillatory behavior during training.
\end{itemize}

The high-dimensional nature of these applications makes second-order methods impractical. The derivative-free subspace structure of the projected line search subspace,  {\tt PLS-S}, framework is therefore well suited for such problems, as it directly targets the NME formulation using only residual evaluations, without requiring explicit Jacobian information. Consequently, projection-based methods, extragradient schemes, and optimistic gradient methods have attracted considerable attention in recent years for solving large-scale NME arising from min--max problems. 
Motivated by these developments, we consider a variety of min--max problems formulated as NME to assess the efficiency and robustness of the proposed {\tt PLS-S} framework compared to state-of-the-art methods. 

The derivative-free subspace structure and projection-based geometry of {\tt PLS-S} make it particularly suitable for machine learning settings in which Jacobian information is unavailable or prohibitively expensive to compute.
 Although NME are not derivative-free optimization (DFO) problems in general and the monotone operator defining the system is often available in analytical form, the explicit computation or storage of Jacobian information becomes prohibitively expensive or even infeasible in large-scale applications. This is particularly the case in modern machine learning, equilibrium modeling, and large discretized systems, where the monotone operator may involve complex composite mappings
or high-dimensional data dependencies. As a result, iterative methods that avoid explicit derivative information and rely solely on residual evaluations are of significant practical interest. The most popular DFO methods are direct-search methods \cite{audet2006mads,conn2009dfo}, line-search methods \cite{VRBBO,VRDFON,SSDFO,SDBOX}, model-based methods \cite{BOBYQA,UOBYQA}, and evolutionary strategies such as matrix adaptation \cite{MADFO,MATRS}. For nonlinear systems in particular, the existing Jacobian-free optimization strategies---a class of DFO methods---include model-based approaches \cite{kimiaei2017new,LMLS}, Jacobian-free Newton-like \cite{deuflhard2011newton,griewank1981modification}, and quasi-Newton schemes \cite{sorber2012unconstrained}.

 In the context of solving NME, the existing iterative strategies can be broadly grouped into two groups: methods employing \textit{fixed step sizes} and methods employing \textit{line-search step sizes}. In the following, we briefly review some methods related to this work.   

\paragraph{\bfi{(i) Methods with fixed step size.}} The simplest iterative strategy for obtaining the solution of NME mimics the classical gradient descent algorithm, regardless of the initialization, and uses a given fixed step size. For convergence with explicit numerical methods for NME, an operator-corrector term is required. Let us briefly discuss some well-known discrete methods with fixed step size to solve NME governed by Lipschitz continuity. The global convergence or last/best-iterate convergence rate of the methods was established under certain boundedness conditions for a fixed step size related to the Lipschitz constant. 

The extragradient ({\tt EG}) method, proposed by Korpelevich  \cite{korpelevich1976extragradient} and Antipin \cite{antipin1997equilibrium}, is a certain modification of the gradient method with a global convergence property that explores the idea of extrapolation. The last-iterate convergence rate of the Euclidean residual norm for this method was established in \cite{gorbunov2022stochastic}. The optimistic gradient descent ascent ({\tt OGDA}) method proposed by Popov \cite{popov1980modification} is formulated in terms of the monotone operator and the idea of momentum. For any two initial points, global convergence for {\tt OGDA} was established. Moreover, the best iterate convergence rate of the residual norm was proved by Chavdarova et al. \cite{chavdarova2021last}. 
Bot et al. \cite{Bo2023} proposed the fast optimistic gradient descent ascent method, which is an accelerated version of the {\tt OGDA}, with the last-iterate convergence rate to a solution. The extra anchored gradient ({\tt EAG}) \cite{yoon2021accelerated} method is an accelerated algorithm to solve NME, designed by using anchor variables, a technique that can be traced back to Halpern’s algorithm \cite{halpern1967fixed}. A dynamic alternative version for {\tt EAG} addresses some issues of the fixed version. However, its last-iterate convergence rate was proven in both cases for any initialization.

\paragraph{\bfi{(ii) Methods with line search step size.}} Projected line search method ({\tt PLS}) uses a line search technique \cite{nocedal2006numerical} combined with the projection strategy proposed by Solodov and Svaiter \cite{SS}. Line search methods are a prominent class of iterative methods to solve unconstrained minimization problems. 
These methods produce new iterates utilizing a suitable step size after determining proper directions for minimization. Effective choices of both direction and step size are key to success and robustness
of the line search approach.  An attractive property of {\tt PLS} is that, under the assumption of continuity and monotonicity of the nonlinear system of equations, the sequence of points generated by {\tt PLS}, starting from an arbitrary initial point, converges to a solution of that system, without requiring the boundedness of the solution set for global convergence. In addition, the distance between the sequence and its solution set decreases, ensuring convergence for {\tt PLS}. For the first time, \cite{SS} introduced a {\tt PLS} using Newton's method for determining the descent direction to solve NME. Recently, for six line search schemes in the existing {\tt PLS} methods, \cite{Ou2022} proposed a framework that provides a unified convergence analysis of the derivative-free projection-based method for NME. The sublinear convergence rate of the sequence to a solution is proven under some conditions.  

\paragraph{\bfi{Subspace directions.}}\label{s.dfd}
There are many different subspace directions along which {\tt PLS} tries to handle NME in low to high dimensions; see, e.g., \cite{aho,dai,HP,XZ}, while the Jacobian-based methods (e.g., the tensor
method \cite{BS}, the Newton method \cite{SS}, the quasi-Newton
method \cite{BKL},  the limited-memory quasi-Newton method \cite{LMLS}, the Gauss-Newton method \cite{LF}, the
Levenberg-Marquardt method \cite{KYF} and their alternative forms)
handle the problem in low to medium dimensions. These subspace directions have the same structure as the conjugate gradient (\texttt{CG}) direction (e.g., \cite{al2,al1},
\cite{Andrei22013,Andrei12013}, \cite{bab1,bab2}, \cite{HZ0}), but instead of the exact gradient, the residual vector is used. We call these directions {\it Jacobian-free subspace directions} (\texttt{JFSdir}), since they are spanned by the residual vector and the scaled previous direction and do not involve derivative information of the monotone equation.

\section{Our contribution}\label{sec:cont}

In this paper, we propose a projection-based subspace framework for solving
nonlinear monotone equations. The framework has two step-size
realizations: the line-search realization, denoted by \texttt{PLS-S}, and its
fixed-step counterpart, denoted by \texttt{PF-S}. Throughout the paper, when
a result applies to both realizations, we refer collectively to them as the
proposed projection-based subspace framework or simply the
proposed methods; the notation \texttt{PLS-S} or \texttt{PF-S} is
used when the distinction between the two step-size strategies is relevant.
Their main ingredients are: (i) an exact projection mechanism, (ii) four
groups of Jacobian-free subspace directions, (iii) an enforced residual-based
angle condition, (iv) a line-search step-size strategy, (v) a fixed step-size
strategy, (vi) asymptotic and canonical \texttt{OGDA}-type direction
structures, (vii) residual--memory representability and local residual decay,
(viii) convergence and complexity analysis, and (ix) numerical results. The
proposed methods use conjugate-gradient-type recursions to generate candidate
subspace directions without evaluating Jacobian matrices, while the effective
directions are controlled through an explicit scaling rule and an angle
safeguard.

This paper is accompanied by supplementary material \cite{suppMat}, which
provides the detailed development underlying contributions~(vi) and~(vii).
In particular, it examines the asymptotic and canonical
\texttt{OGDA}-type direction structures associated with the proposed
subspace directions, establishes their residual--memory representations, and
relates these structures to representative classical \texttt{OGDA}-type
methods. It also discusses the corresponding convergence-rate implications
under additional Fej\'er-type decrease and local error-bound conditions. For
completeness, \cite{suppMat} further provides additional background on
monotone equations and min--max problems, the complete construction and
validation of the learning-based test problems, and detailed numerical
comparisons among the \texttt{PLS-S} and its version \texttt{PF-S} with a fixed step size, and
\texttt{OGDA}-type variants.

The main features of {\tt PLS-S} and \texttt{PF-S} are as follows:\medskip\\
(i) \bfi{Projection method.}
The proposed framework is built upon a projection mechanism generated by the monotonicity of the operator. At each iteration, a trial point
defines a separating hyperplane and an associated closed half-space containing
the solution set. The next iterate is obtained by projecting the current
iterate onto this half-space. This
mechanism guarantees Fej\'er monotonicity of the iterates with respect to the
solution set and plays a central role in establishing global convergence
without requiring a merit function or Jacobian information.\medskip\\
(ii) \bfi{Subspace directions.}
We introduce four groups of Jacobian-free subspace directions, denoted
collectively by {\tt JFSdir}. The associated search subspaces are generated by
two-term and three-term combinations of the current residual, previous
directions, residual differences, and displacement vectors. These
constructions incorporate structured memory into the search process while
preserving the Jacobian-free nature of the method. Importantly, the proposed
subspace directions admit a unified residual--memory interpretation of
\texttt{OGDA} type, in which the negative current residual is corrected
by a controlled term involving information from previous iterations.

This memory correction can mitigate the classical
\bfi{zigzagging behavior} often observed in first-order methods, particularly
on ill-conditioned monotone problems. In our proposed framework,
the residual--memory structure is combined with the projection step, which
controls the distance to the solution set. The convergence analysis is
independent of the particular algebraic formulas used to determine the
subspace coefficients, provided that the resulting effective directions
satisfy the prescribed scaling and angle conditions. This permits a unified
analysis of a broad class of Jacobian-free subspace directions.\medskip\\
(iii) \bfi{Enforced residual-based angle condition.}
A key ingredient of the proposed framework is the enforced angle condition,
which guarantees that the angle between the effective search direction and
the negative residual remains uniformly acute. This geometric condition
ensures sufficient descent independently of the magnitude and algebraic form
of the raw subspace direction generated by the conjugate-gradient-type
recursion.

The effective search direction is additionally controlled through an explicit
scaling rule that bounds its norm above and below by fixed multiples of the
current residual norm. Consequently, global convergence and convergence-rate results can be
established without imposing boundedness assumptions on the raw subspace
directions beyond the safeguarded boundedness of the updating coefficients
specified in property~{\rm (P1)} (defined by \eqref{e.P1}, below). This differs from many
analyses of conjugate-gradient-type and subspace methods, in which boundedness
of the generated directions or parameters is assumed directly.\medskip\\
(iv) \bfi{Line search step sizes.}
The line-search step-size implementation of {\tt PLS-S} employs a
projection-based line search along the effective subspace direction. Unlike Wolfe- or Armijo-type procedures \cite{nocedal2006numerical} and the improved Goldstein procedure \cite{CLS}, this strategy does not rely on Jacobian information or on the decrease of a merit function. Under the Lipschitz
continuity assumption and the enforced angle and scaling conditions, the line-search condition is well defined. When the Lipschitz constant is known,
the safeguard parameter can be selected explicitly so that the accepted step
sizes possess a uniform positive lower bound. This property is essential in
the convergence-rate and residual-evaluation complexity analyses.\medskip\\
(v) \bfi{Fixed step size.}
In addition to the line-search step-size strategy, {\tt PF-S} admits a fixed
step-size implementation. If the fixed step size satisfies an admissibility
condition depending on the Lipschitz constant, the line-search parameter, and
the enforced angle constant, then the sufficient descent inequality and the
projection estimates required by the convergence analysis remain valid. This
case avoids backtracking overhead while retaining the global convergence
and local convergence-rate guarantees established for the exact projection
framework.\medskip\\
(vi) \bfi{Asymptotic and canonical \texttt{OGDA}-type direction structures.}
We establish a unifying structural interpretation of the proposed subspace
directions and several classical first-order schemes through an asymptotic
\texttt{OGDA}-type residual--memory representation. Specifically, along
subsequences converging to a solution, the effective directions generated by
{\tt PLS-S} and {\tt PF-S} can be represented asymptotically as the negative current residual
plus a controlled memory correction involving previous residual information
and a remainder that converges to zero along the considered subsequence.\medskip\\
(vii) \bfi{Residual--memory representability and local residual decay.}
We examine the update directions generated by {\tt EG}, {\tt OGDA}, and
{\tt EAG} and identify an analogous canonical residual--memory structure
under the stated asymptotic assumptions. Fast {\tt OGDA} is discussed
separately because its reduction to the canonical form depends on additional
method-specific parameter identities. This places the
considered directions within a common geometric framework based on their
interaction with the negative residual. The result is structural rather than
algorithmic: it does not independently establish convergence of the classical
schemes and does not replace their method-specific Lyapunov analyses or
sharper convergence-rate results.\medskip\\
(viii) \bfi{Convergence and complexity analysis.}
Global convergence of {\tt PLS-S}/{\tt PF-S} is established under monotonicity and
Lipschitz continuity of the operator, together with the enforced angle
condition and the scaling rule imposed on the effective subspace directions.
No cocoercivity or strong monotonicity assumption is required. Under the exact
projection rule, the generated sequence is Fej\'er monotone with respect to
the solution set, and the global convergence mechanism follows the
projection-based approach underlying the {\tt PLS} method of Solodov and
Svaiter \cite{SS}. The analysis does not rely on merit functions, Jacobian
information, or Wolfe-type line searches \cite{nocedal2006numerical}, which are generally unavailable or
ineffective for nonlinear monotone equations
\cite[Chapter~5]{NW}.

The raw conjugate-gradient-type subspace directions may be unbounded; see, for example
\cite{al2,al1,Andrei22013,Andrei12013,bab1,bab2,HZ0,HP}.
Nevertheless, the convergence analysis requires boundedness only of the
effective scaled directions relative to the residual and does not assume
boundedness of the raw directions themselves. Under monotonicity and Lipschitz continuity, the exact projection variant
satisfies the baseline best-iterate estimate $\min_{0\le k\le\ell}\|E(x_k)\|
=
O(\ell^{-1/2})$. Denoted by $\operatorname{dist}(x,C):=\inf\{\|x-y\|:y\in C\}$ is the distance from a point $x$ to a nonempty closed set $C$ and by $X^*$ is the solution set of \gzit{e.ME}. Then, we define the distance sequence by
\begin{equation}\label{e.defdell}
 d_\ell:=\dist(x_\ell,X^*).  
\end{equation} 
If a local error bound holds, then $d_\ell=o(\ell^{-1/2})$.
Under the stronger local Jacobian regularity condition at the limit solution,
which in particular implies the required local error bound and local isolation
of the limit solution, the iterates converge locally
$R$-linearly, and the last-iterate residual satisfies $\|E(x_\ell)\|=o(\ell^{-1})$. These results apply to both the line-search \texttt{PLS-S} and fixed-step
\texttt{PF-S} realizations, subject to their respective step-size admissibility
conditions.

We also derive explicit iteration and residual-evaluation complexity bounds. The baseline best-iterate guarantee requires $O(\varepsilon^{-2})$
iterations to produce an iterate whose residual does not exceed
$\varepsilon$, whereas the local linear regime yields the pointwise bound $O\!\left(\log\varepsilon^{-1}\right)$. For the line-search step-size implementation, a uniform bound is additionally
obtained for the number of residual evaluations required by the backtracking
procedure at each iteration.\medskip\\
(ix) \bfi{Numerical results.}
We report numerical experiments on a collection of benchmark problems,
including learning-based models formulated as monotone equations. Both the
line-search {\tt PLS-S} methods and their fixed-step {\tt PF-S} counterparts
are compared, and the most competitive variants are further assessed against
{\tt EG}, {\tt OGDA}, and related accelerated schemes
\cite{korpelevich1976extragradient,Malitsky2019,popov1980modification}.
The experiments evaluate robustness and computational efficiency in terms of
solved problems, residual evaluations, and execution time, together with
convergence behavior on representative learning-based problems. The detailed
comparisons reported in \cite{suppMat} show that several proposed variants are
fully robust on the complete test set, with the fixed-step variants being
particularly competitive in residual-evaluation efficiency.

\section{Our proposed projection-based subspace framework}\label{sec.new}
We consider the NME defined by \eqref{e.ME}, that is $E(x) = 0$, $x\in \mathbb{R}^n$, and
the mapping $E: \mathbb{R}^n \to \mathbb{R}^n$ is continuous and monotone, i.e.,
\eqref{e.monotone} holds.

\subsection{Projected line search for the NME}\label{sub.pls} 

We now describe the projected line-search framework applied to the
original operator \(E\). Given a search direction \(p_\ell\) satisfying
the descent condition
\begin{equation}\label{e.des}
p_\ell^{T}E(x_\ell) < 0,
\end{equation}
like the classical {\tt PLS} in \cite{SS}, the {\tt PLS-S} method computes the trial point
\begin{equation}\label{e.xbar}
\bar{x}_\ell := x_\ell + \mu_\ell p_\ell,
\end{equation}
where $\mu_\ell>0$ is chosen such that the sufficient descent
condition holds:
\begin{equation}\label{e.bls}
p_\ell^{T}E(\bar{x}_\ell)
\le
-\rho\,\mu_\ell\|p_\ell\|^{2},
\qquad \rho\in(0,1).
\end{equation}
Using \gzit{e.xbar} and \gzit{e.bls}, we obtain 
\begin{equation}\label{e.rbls}
(x_\ell-\bar{x}_\ell)^T E(\bar{x}_\ell)
=
-\mu_\ell p_\ell^T E(\bar{x}_\ell)
\ge
\rho\,\mu_\ell^2\|p_\ell\|^2
>
0.
\end{equation}
We assume throughout that the solution set $X^*$ of \gzit{e.ME} is nonempty; we refer to this as the {\it assumption {(\Asol)}}. Consequently, for every $x^*\in X^*$, monotonicity of $E$ together with $E(x^*)=0$ implies
\[
\big(E(\bar{x}_\ell)-E(x^*)\big)^T
(\bar{x}_\ell-x^*)
\ge0.
\]
Hence, $(x^*-\bar{x}_\ell)^T E(\bar{x}_\ell)
\le0$. Therefore, the hyperplane
\begin{equation}\label{e.Hell}
H_\ell :=
\{x\in\mathbb{R}^n :
(x-\bar{x}_\ell)^T E(\bar{x}_\ell)=0\}
\end{equation}
separates the current iterate $x_\ell$ from the original solution set
$X^*$. Indeed, using \eqref{e.xbar} and \eqref{e.bls}, we obtain
\begin{equation}
(x_\ell-\bar{x}_\ell)^T E(\bar{x}_\ell)>0,
\label{eq:positive-inner-product}
\end{equation}
whereas $(x^*-\bar{x}_\ell)^T E(\bar{x}_\ell)\le0 $ for all $x^*\in X^*$. The half-space
\begin{equation}\label{e:semspa}
\ol H_\ell :=
\{x\in\mathbb{R}^n :
(x-\bar x_\ell)^T E(\bar x_\ell)\le 0\}
\end{equation}
and the new iterate 
\begin{equation}\label{eq:projection-step}
x_{\ell+1}:=
P_{\ol H_\ell}(x_\ell)
\end{equation}
are defined. From
\eqref{e.monotone}, $(x^*-\bar x_\ell)^T E(\bar x_\ell)\le0$, so that
\begin{equation}\label{e.XH}
X^*\subseteq \ol H_\ell;
\end{equation}
hence, the projection
set is nonempty. Since \eqref{eq:positive-inner-product} implies that
$x_\ell\notin \ol H_\ell$, this projection reduces to the
orthogonal projection onto the separating hyperplane. The
projection parameter is
\begin{equation}\label{e.lamb}
\lambda_\ell
:=
\frac{(x_\ell-\bar{x}_\ell)^T E(\bar{x}_\ell)}
{\|E(\bar{x}_\ell)\|^{2}},
\end{equation}
and the new iterate is given explicitly by
\begin{equation}\label{e.proj}
x_{\ell+1}
:=
P_{\ol H_\ell}(x_\ell)=
x_\ell
-
\lambda_\ell E(\bar{x}_\ell),
\end{equation}
which is the Euclidean projection of $x_\ell$ onto $\ol H_\ell$.

\subsection{Multi-Term Subspace Directions with Controlled Scaling}\label{sub.subdir}
In this subsection, we construct and analyze a family of multi-term
Jacobian-free subspace directions that form the search mechanism of the
{\tt PLS-S}/{\tt PF-S} framework. The objective is twofold. First, we introduce
two-term and three-term directions that incorporate structured memory
information from previous iterates, thereby enriching the search subspace
beyond the current residual direction. Second, we equip these directions
with an explicit scaling rule that controls their magnitude relative to
the original residual \(E(x_\ell)\). This control is essential
for establishing the subsequent convergence results independently of the
specific algebraic form of the updating coefficients.

The resulting directions are structurally inspired by classical CG-type
updates, but they differ from traditional CG directions in several
essential aspects. In particular, we do not impose conjugacy conditions
and do not employ Armijo- or Wolfe-type line searches. Instead, the step
size is determined by the projected line-search mechanism described in
Section~\ref{sub.pls}.  

For an arbitrary initial point \(x_0\in\mathbb{R}^n\), we define
\begin{equation}\label{e.sy}
s_{\ell-1}
=
\bar{x}_{\ell-1}-x_{\ell-1},
\qquad
y_{\ell-1}
=
E(\bar{x}_{\ell-1})
-
E(x_{\ell-1}),
\end{equation}
where \(\bar{x}_{\ell-1}\) is given by \eqref{e.xbar}. The vectors \(s_{\ell-1}\) and \(y_{\ell-1}\) contain step and
residual-difference information from the previous iteration. The directions
\texttt{JFSdir} are divided into the following four groups:
\begin{equation}\label{e.pcg1}
p_\ell=
\left\{
\begin{array}{lcll}
\operatorname{span}\big(E(x_\ell),p_{\ell-1}\big)
&=&
- E(x_\ell)
+ \beta_\ell p_{\ell-1}
& \text{({\tt JFS2dir})},
\\[8pt]

\operatorname{span}\big(E(x_\ell),p_{\ell-1},y_{\ell-1}\big)
&=&
- E(x_\ell)
+ \beta^1_\ell p_{\ell-1}
+ \beta^2_\ell y_{\ell-1}
& \text{({\tt JFS3dir1})},
\\[8pt]

\operatorname{span}\big(E(x_\ell),y_{\ell-1},s_{\ell-1}\big)
&=&
- E(x_\ell)
+ \beta^3_\ell y_{\ell-1}
+ \beta^4_\ell s_{\ell-1}
& \text{({\tt JFS3dir2})},
\\[8pt]

\operatorname{span}\big(E(x_\ell),p_{\ell-1},h_{\ell}\big)
&=&
- E(x_\ell)
+ \beta_\ell p_{\ell-1}
+ \beta^5_\ell h_{\ell}
& \text{({\tt JFS3dir3})}.
\end{array}
\right.
\end{equation}
The directions in \eqref{e.pcg1} are candidate directions generated
from the current residual and the available memory vectors without evaluating gradients or Jacobian matrices. For this reason, we refer to them as \texttt{JFSdir} rather than CG directions. Before being used in the projected line search, each candidate direction is, when necessary, modified so as to satisfy the angle condition (defined by \gzit{e.rac}, below) with respect to \(E(x_\ell)\), and the resulting direction is then scaled by \gzit{e.scaleDir}.

We define the coefficients $\beta_\ell$ and $\beta^i_\ell$, $i=1,\ldots,5$, in Tables~\ref{tab:jfs2}--\ref{tab:jfs3dir3} of Appendix~\ref{app:table}, where we also establish their well-definedness. We here explain these groups as follows:\medskip\\
(i) \bfi{Group \texttt{JFS2dir} of subspace directions with two terms.} This group has the form of the \texttt{CG} directions of Fletcher--Reeves (\texttt{FR})~\cite{flre}, Polak--Ribi\'ere (\texttt{PR})~\cite{pol}, Hestenes--Stiefel (\texttt{HS})~\cite{hes}, Dai--Yuan (\texttt{DY})~\cite{dai99}, Liu--Storey (\texttt{LS})~\cite{liu1}, Dai--Liao (\texttt{DL})~\cite{dl}, and Hager--Zhang (\texttt{HZ})~\cite{HZ0}.\medskip\\
(ii) \bfi{First group \texttt{JFS3dir1} of subspace directions with three terms.} This group has the form of the \texttt{CG} directions of Zhang et al.~\cite{ZZL1,ZZL2} ({\tt Z1} and {\tt Z2}).\medskip\\
(iii) \bfi{Second group \texttt{JFS3dir2} of subspace directions with three terms.} This group has the form of the \texttt{CG} directions of Andrei~\cite{Andrei22013,Andrei12013} ({\tt An1} and {\tt An2}) and Deng et al.~\cite{Deng2015} ({\tt De}).\medskip\\
(iv) \bfi{Third group \texttt{JFS3dir3} of subspace directions with three terms.} This group has the form of the \texttt{CG} directions of Al-Baali et al.~\cite{al1}.

\paragraph{Scaling directions.} Let \(p_\ell\neq0\) denote the direction obtained after the angle safeguard.
We scale it relative to the original residual
\(E(x_\ell)\) according to
\begin{equation}\label{e.scaleDir}
p_{\ell}^{\scal}
:=
\eta_{\ell}p_{\ell},
\qquad
\eta_{\ell}
:={
\mathrm{mid}\left\{
\zeta\|E(x_{\ell})\|,
\,
\|p_{\ell}\|,
\,
\|E(x_{\ell})\|
\right\}
}/{
\|p_{\ell}\|
}.
\end{equation}
By construction, the scaled direction satisfies
\begin{equation}\label{e:sccon}
\zeta\|E(x_{\ell})\|
\;\le\;
\|p^{\scal}_{\ell}\|
\;\le\;
\|E(x_{\ell})\|.
\end{equation}

Here, \(0<\zeta<1\) is a tuning parameter and
\(\mathrm{mid}\{a,b,c\}\) denotes the median of the three scalars
\(a,b,c\), namely,
\[
\mathrm{mid}\{a,b,c\}
=
a+b+c
-
\min\{a,b,c\}
-
\max\{a,b,c\}.
\]

\paragraph{Parameter control in the subspace construction.}
The subspace construction is carried out under the following intrinsic design properties.

\textbf{(P1) Controlled updating coefficients.}
The updating parameters in \eqref{e.pcg1} are algorithmic quantities
selected to control the interaction between the current residual and
the memory components. The coefficients
\(\beta_\ell\) and \(\beta_\ell^i\), defined in
Tables~\ref{tab:jfs2}--\ref{tab:jfs3dir3}, are generated by safeguarded
formulas, which satisfy
\begin{equation}\label{e.P1}
\beta_\ell,\beta_\ell^i
\in
[\beta_{\min},\beta_{\max}]
\qquad
\text{for all }\ell
\text{ and }i=1,\ldots,5,
\end{equation}
where $-\infty
<
\beta_{\min}
\le
\beta_{\max}
<
\infty$. The boundedness of these coefficients prevents direct unbounded
amplification by the scalar updating parameters.

\textbf{(P2) Algorithmic control of the scaled direction.}

The definition of \(\eta_\ell\) in \eqref{e.scaleDir} directly enforces
the magnitude condition \eqref{e:sccon}. In particular,
\[
\|p_\ell^{\scal}\|
=
\mathrm{mid}
\left\{
\zeta\|E(x_\ell)\|,
\,
\|p_\ell\|,
\,
\|E(x_\ell)\|
\right\}.
\]
No uniform upper or lower bound on the scalar \(\eta_\ell\) is required
for the convergence analysis; the essential property is the direct norm
control of \(p_\ell^{\scal}\).

\paragraph{\textbf{\texttt{OGDA}-type viewpoint of the subspace directions.}}
The directions defined in \eqref{e.pcg1} contain residual-memory terms
and therefore admit an algebraic interpretation reminiscent of
\texttt{OGDA}-type schemes. In particular, some choices of the memory
vector and updating coefficients combine a current residual
with information derived from previous residuals.

This interpretation is structural only. It does not imply that the
directions reproduce the exact fast-\texttt{OGDA} recurrence or inherit
its last-iterate convergence rate. A precise residual-memory
representation is provided later in \cite[Theroem~1]{suppMat}, after the required convergence and asymptotic-regularity properties of {\tt PLS-S}/{\tt PF-S} have been established.

\subsection{Enforcing angle condition}

In this section, we introduce the angle-safeguarding mechanism that enables
the use of the different subspace directions discussed in
Section~\ref{sub.subdir} by enforcing the angle condition whenever it is violated. The accepted direction must provide uniform descent for the operator
\(E\), since the line search and separating geometry are defined with
respect to \(E\).

Inspired by the gradient-based angle condition proposed in \cite{NA}, we
define the residual-based angle condition
\begin{equation}\label{e.rac}
\frac{(p_\ell)^T E(x_\ell)}
{\|p_\ell\|\,
 \|E(x_\ell)\|}
\le
-\Delta_{\angle}<0,
\hspace{3.0mm}
\text{with}
\hspace{3.0mm}
0<\Delta_{\angle}<1.
\end{equation}
If \(p_\ell\) does not satisfy \eqref{e.rac}, our method modifies it so that the resulting direction satisfies
\eqref{e.rac}. For this purpose, assuming \(E(x_\ell)\neq0\) and
\(p_\ell\neq0\), we define
\begin{eqnarray}\label{e.sigma}
\sigma_1 &:=& E(x_\ell)^T E(x_\ell)>0, \ \
\sigma_2 := (p_\ell)^T p_\ell>0, \nonumber\\
\sigma   &:=& E(x_\ell)^T p_\ell, \ \
\bar{\sigma} := \displaystyle
\frac{\sigma}{\sqrt{\sigma_1\sigma_2}}\in[-1,1]
\end{eqnarray}
and compute
\begin{equation}\label{e.tdef}
t
:=
\D\frac{\sigma+\Delta_{\angle}\sqrt{w}}{\sigma_1}
\in(0,\infty),
\qquad
w
:=
\D\frac{
\sigma_1\sigma_2(1-\bar{\sigma}^2)
}{
1-\Delta_{\angle}^2
}.
\end{equation}
We then define
\begin{equation}\label{e.pcg2}
p^{\modify}_{\ell}
:=
p_\ell-tE(x_\ell).
\end{equation}
If \(w=0\), we set \(p^{\modify}_\ell=-E(x_\ell)\). Otherwise, we use
\eqref{e.pcg2}. When the modification is applied, we then set
\(p_\ell:=p^{\modify}_\ell\); if the angle condition is already satisfied,
\(p_\ell\) is left unchanged. The resulting nonzero direction \(p_\ell\)
is then scaled according to \eqref{e.scaleDir}, yielding
\(p^{\scal}_\ell\).

The following result, which is adapted from Proposition~5.2 in
\cite{NA}, shows that the modification \gzit{e.pcg2} enforces the
required angle condition \gzit{e.rac} with respect to the original residual
\(E(x_\ell)\).

\begin{prop}\label{p.angle}
Suppose that $E(x_\ell)\neq0$, $p_\ell\neq0$, and $0<\Delta_{\angle}<1$. If the candidate direction $p_\ell$ does not
satisfy the angle condition \gzit{e.rac}, then, when \(w>0\), the modified direction
\gzit{e.pcg2} with \(t\) given in \gzit{e.tdef} satisfies \gzit{e.rac} as  
\[
\frac{
(p^{\modify}_\ell)^T E(x_\ell)
}{
\|p^{\modify}_\ell\|\,
\|E(x_\ell)\|
}
=
-\Delta_{\angle}.
\]
If \(w=0\), the fallback direction
\(p^{\modify}_\ell=-E(x_\ell)\) satisfies \gzit{e.rac} with cosine equal to \(-1\).
Moreover, the subsequent scaling \eqref{e.scaleDir} preserves the angle condition since
\(\eta_\ell>0\).
\end{prop}

\begin{proof}
Recall the definitions in \gzit{e.sigma}. From
\eqref{e.pcg2}, we obtain
\[
(p^{\modify}_\ell)^T E(x_\ell)
=
\sigma-t\sigma_1 \ \ \mbox{and} \ \
\|p^{\modify}_\ell\|^2
=
\sigma_2-2t\sigma+t^2\sigma_1.
\]
Therefore, the equality form of the angle condition \eqref{e.rac} is
\begin{equation}\label{e.pr11}
\frac{
\sigma-t\sigma_1
}{
\sqrt{
\sigma_1
\bigl(
\sigma_2-2t\sigma+t^2\sigma_1
\bigr)
}
}
=
-\Delta_{\angle}.
\end{equation}
From the definition of \(t\) in \eqref{e.tdef}, we obtain $\sigma-t\sigma_1
=
-\Delta_{\angle}\sqrt{w}
\le0$. Moreover,
\[
\begin{aligned}
\sigma_1
\bigl(
\sigma_2-2t\sigma+t^2\sigma_1
\bigr)
&=
\sigma_1\sigma_2-\sigma^2
+
(\sigma-t\sigma_1)^2
\\
&=
\sigma_1\sigma_2(1-\bar{\sigma}^2)
+
(\sigma-t\sigma_1)^2.
\end{aligned}
\]
Using $\sigma-t\sigma_1
=
-\Delta_{\angle}\sqrt{w}$ 
and the definition of \(w\), we obtain
\begin{equation}\label{e.pr12}
(1-\Delta_{\angle}^2)
(\sigma-t\sigma_1)^2
=
\Delta_{\angle}^2
\sigma_1\sigma_2
(1-\bar{\sigma}^2).
\end{equation}
Consequently, $\sigma_1
\bigl(
\sigma_2-2t\sigma+t^2\sigma_1
\bigr)
=
w$. Provided that \(w>0\), it follows that
\[
\frac{
\sigma-t\sigma_1
}{
\sqrt{
\sigma_1
\bigl(
\sigma_2-2t\sigma+t^2\sigma_1
\bigr)
}
}
=
\frac{-\Delta_{\angle}\sqrt{w}}{\sqrt{w}}
=
-\Delta_{\angle}.
\]
If \(w=0\), then the candidate direction is collinear with
\(E(x_\ell)\). Since the angle condition is violated, the direction
cannot be a negative multiple of \(E(x_\ell)\). In this degenerate case,
we set $p^{\modify}_\ell=-E(x_\ell)$, which satisfies \eqref{e.rac} with cosine equal to \(-1\).
This completes the proof. \qed
\end{proof}

In finite-precision arithmetic, rounding errors may occasionally produce
\(\bar{\sigma}\notin[-1,1]\). We therefore replace
\(\bar{\sigma}\) by its clipped value, which can be written as
 $\bar{\sigma}_{\rm clip}
:=
\min\{1,\max\{-1,\bar{\sigma}\}\}$.  Using it, we then compute the real value $w
=
{
\sigma_1\sigma_2
\max\{0,1-\bar{\sigma}_{\rm clip}^{\,2}\}
}/{
(1-\Delta_{\angle}^2
)}$.

\subsection{Overview of the proposed projection-based subspace framework}\label{sec:alg}

\begin{algorithm}[http!]
    \caption{Projection-based subspace framework for monotone equations}\label{a.FDS}
    \begin{algorithmic}[1]
        \Statex {\bf Input.} Initial point $x_0\in\mathbb{R}^n$, the maximal number {\tt nRmax} of residual evaluations, minimum
        threshold $\eps\in(0,1)$ on the norm of the residual vector
        \algrule
       
        \Statex {\bf Tuning parameters.} $\Delta_{\angle}\in(0,1)$ (parameter for the angle condition), $\rho\in(0,1)$ (parameter for sufficient descent), $\gamma>1$ (parameter for reducing the line-search step sizes), $\mu_{\init}\in(0,1]$ (initial allowable step size in the line search), $\alpha>0$ (fixed step size for \texttt{PF-S} satisfying
\eqref{e.alpha_bound}), $\beta_{\min},\beta_{\max}\in\mathbb{R}$ (bounds for subspace coefficients), and $\zeta\in(0,1)$ (parameter for the direction scaling)
        \algrule
        \Statex {\bf Output.} An approximate solution of \gzit{e.ME}
        \algrule
        \State ({\sc S0}) Set $\ell:=1$ and $p_0:=p_0^{\scal}:=-E(x_0)$;

        \While{1}

     \State ({\sc S1}) Find a step size $\mu_{\ell-1}\in[\mu_{\min},\mu_{\init}]$ satisfying \eqref{e.bls} with $p_{\ell-1}^{\scal}$, \Comment{({\sc S1$_a$})}
\Statex \spc \spc ~ evaluate $\bar x_{\ell-1}:=x_{\ell-1}+\mu_{\ell-1}p_{\ell-1}^{\scal}$, and compute $E(\bar{x}_{\ell-1})$; instead
\Statex \spc \spc ~ for the fixed-step \texttt{PF-S} realization, set $\mu_{\ell-1}:=\alpha$ satisfying \eqref{e.alpha_bound}, \Comment{({\sc S1$_b$})}
\Statex \spc \spc ~ evaluate $\bar x_{\ell-1}:=x_{\ell-1}+\alpha p_{\ell-1}^{\scal}$, and compute $E(\bar{x}_{\ell-1})$;

           \State  ({\sc S2}) \bfi{Stop} if {\tt nRmax} is reached;
           \State ({\sc S3}) Evaluate the new point by \eqref{eq:projection-step} as $x_\ell:=P_{\ol H_{\ell-1}}(x_{\ell-1})$;
           \vspace{0.001mm}
            \State  ({\sc S4}) \bfi{Stop} if  $\|E(x_{\ell})\| \le \eps$ or {\tt nRmax} is reached;

          \State ({\sc S5}) Compute the subspace direction $p_{\ell}$ by \eqref{e.pcg1}; if $p_\ell=0$, set $p_\ell:=-E(x_\ell)$;
\Statex \spc \spc ~ if $p_\ell$ does not satisfy the angle condition \eqref{e.rac}, replace it by $p_\ell^{\modify}$ as in 
\Statex \spc \spc ~ Proposition~\ref{p.angle}; scale the resulting direction with \eqref{e.scaleDir} to obtain $p_\ell^{\scal}$, and \Statex \spc \spc ~  set $\ell:=\ell+1$;

        \EndWhile
    \end{algorithmic}

\end{algorithm}
An overview of the proposed projection-based subspace framework is given in Algorithm~\ref{a.FDS}. It includes six steps, ({\sc S0})–({\sc S5}). Step~({\sc S0}) is performed once for $\ell = 1$, and Steps~({\sc S1})–({\sc S5}) are then repeated for $\ell> 1$. The framework includes both the line-search
\texttt{PLS-S} realization and the fixed-step \texttt{PF-S} realization,
which differ only in the step-size strategy used in Step~{\rm({\sc S1})}.

({\sc S0}) is an initialization step, setting $p_0=p^{\scal}_0=-E(x_0)$. If $E(x_0)=0$, then $x_0$ is a solution of \eqref{e.ME} and the algorithm terminates. Otherwise, $p^{\scal}_0$ satisfies both the scaling condition \eqref{e:sccon} and the angle condition \eqref{e.rac}.  

({\sc S1}) is a step-size selection (variant or fixed) step. Given $x_{\ell-1}$ and $p^{\scal}_{\ell-1}$, the trial point is defined by $\bar x_{\ell-1}=x_{\ell-1}+\mu_{\ell-1}p^{\scal}_{\ell-1}$, 
where the step size $\mu_{\ell-1}>0$ is chosen according to one of the
following two strategies:

({\sc S1$_a$}) {Line search step size.}
Set $\mu_{\ell-1}=\mu_{\init}\in(0,1]$ as an initial step size.
While the sufficient descent condition \eqref{e.bls} is not satisfied and
$\mu_{\ell-1}>\mu_{\min}$, update
$\mu_{\ell-1}\leftarrow \max\{\mu_{\ell-1}/\gamma,\mu_{\min}\}$, recompute $\bar x_{\ell-1}$ and the residual 
$E(\bar x_{\ell-1})$. The safeguard $\mu_{\min}$ is selected so that
\begin{equation}
0<\mu_{\min}
\le
\min\left\{
\mu_{\init},
{\Delta_{\angle}}/{(L+\rho)}
\right\},
\label{eq:mu-min-bound}
\end{equation}
where $L$ is the Lipschitz constant of $E$. Consequently, the sufficient descent condition \eqref{e.bls} is satisfied no later than the trial $\mu_{\ell-1}=\mu_{\min}$.
Once \eqref{e.bls} holds, the new iterate is obtained by \eqref{e.proj}.

Under the assumption {\rm (\Alip)}, i.e., that $E$ is Lipschitz continuous with constant $L>0$, the line search therefore admits a uniform positive lower bound on the accepted step size. The parameter $\mu_{\min}$ serves as a numerical safeguard and provides a uniform bound on the maximum number of backtracking reductions.

({\sc S1$_b$}) {Fixed step size.}
Choose a constant step size $\mu_{\ell-1}=\alpha$ satisfying
\begin{equation}\label{e.alpha_bound}
0<\alpha \le \min\{1,{\Delta_{\angle}}/{(L+\rho)}\}.
\end{equation}
Then, the trial point is defined by $\bar x_{\ell-1}=x_{\ell-1}+\alpha\,p^{\scal}_{\ell-1}$. The sufficient descent condition
\eqref{e.bls} holds with $\mu_{\ell-1}=\alpha$ (see Lemma~\ref{le:fixed-descent-modified}, below) and the new iterate is obtained
directly by the projection step \eqref{e.proj}. 

Note that later (see Lemma~\ref{le:lowerstep}, below), we use the specific choice
\begin{equation}\label{e.mu-min}
\mu_{\min}
=
\min\left\{
\mu_{\init},{\Delta_{\angle}}/{(\gamma(L+\rho))}
\right\},
\end{equation}
which satisfies $0<\mu_{\min}\le\mu_{\init}\le1$ and the general admissibility condition \eqref{eq:mu-min-bound} since $\gamma>1$.

({\sc S2}) is the first termination criterion. At this stage, the residual
$E(\bar{x}_{\ell-1})$ has already been evaluated at the trial point.
If the maximum number {\tt nRmax} of residual evaluations is attained, the
algorithm stops because the prescribed computational budget has been exhausted.

({\sc S3}) is the projection point step. It computes the projection parameter and the new iterate. Since \eqref{e.xbar} and \eqref{e.bls} give $(x_{\ell-1}-\bar{x}_{\ell-1})^TE(\bar{x}_{\ell-1})>0$, 
we necessarily have $E(\bar{x}_{\ell-1})\neq0$. Hence, the projection parameter \eqref{e.lamb} is well-defined, so the explicit formula \eqref{e.proj} is used. 

({\sc S4}) is the second termination criterion. At this step, the algorithm terminates if either $\|E(x_{\ell})\| \le \eps$, in which case $x_{\ell}$ is accepted as an $\eps$-approximate solution of \eqref{e.ME}, or the maximum number {\tt nRmax} of residual evaluations has been reached, indicating that the prescribed computational budget has been exhausted.

({\sc S5}) computes a scaled subspace direction.
It computes the quantities in \eqref{e.sy} and the candidate subspace direction $p_\ell$ by \eqref{e.pcg1}, where the parameters
$\beta_\ell$ and $\beta_\ell^i$, $i=1,\ldots,5$, are defined as in
Tables~\ref{tab:jfs2}--\ref{tab:jfs3dir3} of Appendix \ref{app:table} and restricted by \eqref{e.P1}. If the candidate direction satisfies $p_\ell=0$, it is replaced by $p_\ell=-E(x_\ell)$. If the resulting nonzero candidate direction does not satisfy the angle condition \eqref{e.rac}, it is modified according to \eqref{e.sigma}--\eqref{e.pcg2}, and the modified direction is again denoted by $p_\ell$. The resulting direction is then scaled by \eqref{e.scaleDir}. Since this scaling multiplies the direction by a positive scalar, it preserves the angle condition \eqref{e.rac}.

Under the assumption {\rm (\Alip)} on a set containing all iterates and trial points generated by the algorithm, the following lemma estimates the residual along a scaled direction satisfying the scaling and angle conditions.

\begin{lem} \label{lem:lipschitz-direction} Suppose that {\rm (\Alip)} holds and that
the scaled direction $p_\ell^{\scal}$ satisfies \eqref{e:sccon} and \eqref{e.rac}, i.e., 
\begin{equation}
(p_\ell^{\scal})^T E(x_\ell)
\le
-\Delta_{\angle}\|p_\ell^{\scal}\|^2,
\label{eq:angle-descent-bound}
\end{equation}
and recall $\bar x_\ell:=x_\ell+\mu p_\ell^{\scal}$ with $\mu>0$ from \gzit{e.xbar}. Then
\begin{equation}\label{e.lipschitz-direction}
(p_\ell^{\scal})^T E(\bar x_\ell)
\le
-\bigl(\Delta_{\angle}-L\mu\bigr)
\|p_\ell^{\scal}\|^2.
\end{equation}
\end{lem}

\begin{proof}
By the Lipschitz continuity of $E$ and the Cauchy--Schwarz inequality,
\[
\begin{aligned}
(p_\ell^{\scal})^T E(\bar x_\ell)
&=
(p_\ell^{\scal})^T E(x_\ell)
+
(p_\ell^{\scal})^T
\bigl(
E(\bar x_\ell)-E(x_\ell)
\bigr)
\\
&\le
-\Delta_{\angle}
\|p_\ell^{\scal}\|^2
+
\|p_\ell^{\scal}\|
\|E(\bar x_\ell)-E(x_\ell)\|
\\
&\le
-\Delta_{\angle}
\|p_\ell^{\scal}\|^2
+
L\mu
\|p_\ell^{\scal}\|^2
=
-\bigl(
\Delta_{\angle}-L\mu
\bigr)
\|p_\ell^{\scal}\|^2.
\end{aligned}
\]
\qed
\end{proof}

In the following result, we show that the line search procedure in Step~({\sc S1$_a$}) of Algorithm \ref{a.FDS} terminates in a finite number of iterations.

\begin{prop}\label{pr0}
Suppose that {\rm (\Alip)} holds, the scaled direction satisfies \eqref{e:sccon} and \eqref{e.rac}, and $\mu_{\min}$ satisfies \gzit{eq:mu-min-bound}. Then the line search procedure in Step~{\rm({\sc S1$_a$})} of Algorithm \ref{a.FDS} is
well-defined. At every nonstationary iteration, it terminates after at
most $m_{\max}
:=
\left\lceil
\log_{\gamma}
\left({\mu_{\init}}/{\mu_{\min}}
\right)
\right\rceil$ step-size reductions. Since the initial trial is also evaluated, the
number of trial residual evaluations of the line search procedure is at most $m_{\max}+1
=
1+
\left\lceil
\log_{\gamma}
\left({\mu_{\init}}/{\mu_{\min}}
\right)
\right\rceil$.
\end{prop}

\begin{proof}
Let $\ell_0\ge0$ be a nonstationary iteration, so that $E(x_{\ell_0})\neq0$. From \eqref{eq:angle-descent-bound}, $(p_{\ell_0}^{\scal})^T
E(x_{\ell_0})
\le
-\Delta_{\angle}
\|p_{\ell_0}^{\scal}\|^2$. Let $\mu>0$ and define $\bar x_{\ell_0}
=
x_{\ell_0}
+
\mu p_{\ell_0}^{\scal}$. By Lemma \ref{lem:lipschitz-direction}, $(p_{\ell_0}^{\scal})^T
E(\bar x_{\ell_0})
\le
-\big(
\Delta_{\angle}-L\mu
\big)
\|p_{\ell_0}^{\scal}\|^2$. Therefore, if $\mu
\le{\Delta_{\angle}}/{(L+\rho)}$, then $\Delta_{\angle}-L\mu
\ge
\rho\mu$, and hence $(p_{\ell_0}^{\scal})^T
E(\bar x_{\ell_0})
\le
-\rho\mu
\|p_{\ell_0}^{\scal}\|^2$. 
Thus, the sufficient descent condition \eqref{e.bls} holds for every $0<\mu
\le{\Delta_{\angle}}/{(L+\rho)}$. Since $\mu_{\min}
\le {\Delta_{\angle}}/{(L+\rho)}$, the line search must accept a step size no later than the trial $\mu=\mu_{\min}$. It therefore terminates with a positive step size.

It remains to estimate the maximal number of reductions.
At each unsuccessful trial, the step size is reduced according to $\mu
\leftarrow
\max\left\{{\mu}/{\gamma},
\mu_{\min}
\right\}$. Starting from $\mu_{\init}$, after $m$ reductions before the safeguard is activated, we have $\mu
={\mu_{\init}}/{\gamma^m}$. The safeguard is reached once ${\mu_{\init}}/{\gamma^m}
\le
\mu_{\min}$. Equivalently, $\gamma^m
\ge{\mu_{\init}}/{\mu_{\min}}$. Therefore, the number of reductions is bounded by
\[
m_{\max}
:=
\left\lceil
\log_{\gamma}\!\left({\mu_{\init}}/{\mu_{\min}}\right)
\right\rceil,
\]
which is finite and independent of $\ell_0$. \qed
\end{proof}

\section{Global convergence analysis}\label{s.conv}

We now establish the global convergence properties of Algorithm \ref{a.FDS}.
The analysis relies on the monotonicity of the original operator, the separating
hyperplane geometry induced by the projected line search, and a uniform descent
property of the trial points.
Both the line-search step size and fixed step-size forms are covered within a
unified framework.

To establish the global convergence of Algorithm \ref{a.FDS}, we
assume (\Asol) and (\Alip) that are explained in Sections \ref{sub.pls} and \ref{sec:alg}, respectively.

The following lemma shows that, for a sufficiently small fixed step size,
the trial point $\bar x_\ell$ satisfies the sufficient descent inequality
\eqref{e.bls}, without invoking any line search.

\begin{lem}[Fixed step sufficient descent]
\label{le:fixed-descent-modified}
Suppose {\rm (\Alip)} holds and the scaled search direction
$p_\ell^{\scal}$ satisfies the scaling condition \eqref{e:sccon} and the enforced angle condition \eqref{e.rac}. Recall the trial point \gzit{e.xbar} with $p_\ell=p^{\scal}_\ell$ as $\bar x_\ell$. Assume the fixed step size $\alpha$ satisfies \eqref{e.alpha_bound}. Then, for all $\ell \ge 0$, the sufficient descent condition
\begin{equation}\label{e:fixed-bls-mod}
(p^{\scal}_\ell)^{T} E(\bar x_\ell)
\le
-\rho\alpha\,\|p_\ell^{\scal}\|^{2}
\end{equation}
holds. Consequently, \eqref{e:fixed-bls-mod} is precisely
\eqref{e.bls} with $\mu_\ell=\alpha$.
\end{lem}

\begin{proof}
From \eqref{eq:angle-descent-bound}, $(p_\ell^{\scal})^T E(x_\ell)
\le
-\Delta_{\angle}\,
\|p_\ell^{\scal}\|^2$ holds. Using {\rm (\Alip)}, \eqref{e.xbar}, and Lemma \ref{lem:lipschitz-direction}, $(p_\ell^{\scal})^T E(\bar x_\ell)
\le
-\big(\Delta_{\angle}-L\alpha\big)
\|p_\ell^{\scal}\|^2$.
From \eqref{e.alpha_bound}, $\alpha
\le {\Delta_{\angle}}/{(L+\rho)}$, so that $\Delta_{\angle}-L\alpha
\ge
\rho\alpha$. This proves \eqref{e:fixed-bls-mod}. \qed
\end{proof}

Inspired by \cite[Theorem~2.1]{SS}, we have, by the following lemma, the
boundedness of some sequences generated via Algorithm \ref{a.FDS}. Moreover,  using \eqref{e.defdell}, the projection inequality (see Proposition~\ref{p.Projection} in Appendix~\ref{app:Relaxed}) gives
\begin{equation}
d^2_{\ell+1} \le d^2_\ell -\|x_{\ell+1}-x_\ell\|^2 \ \ \text{or} \ \ \|x_{\ell+1}-x^*\|^2
\le
\|x_\ell-x^*\|^2
-
\|x_{\ell+1}-x_\ell\|^2
\label{eq:fejer-descent}
\end{equation}
and, in a special case, a separating argument shows that the distance to the solution set for the
constructed sequence decreases monotonically, which
provides the main Fej\'er-monotonicity estimate for the global convergence
analysis of our algorithm. The argument applies to both the line-search
step size and fixed step-size types of Algorithm \ref{a.FDS}, since in both cases the trial point satisfies a sufficient descent inequality.

\begin{lem}\label{le:bo}
Suppose that {\rm (\Asol)} and {\rm (\Alip)} hold and that Algorithm \ref{a.FDS} with either the line-search step size, discussed in {\rm({\sc S1$_a$})}, or the fixed step size $\mu_\ell=\alpha$, discussed in {\rm({\sc S1$_b$})}, satisfying
\eqref{e.alpha_bound} generates the sequences
$\{x_\ell\}_{\ell\geq0}$,
$\{\bar{x}_\ell\}_{\ell\geq0}$,
$\{E(x_\ell)\}_{\ell\geq0}$,
and $\{E(\bar{x}_\ell)\}_{\ell\geq0}$,
where $\bar{x}_\ell$ is defined by \eqref{e.xbar} as
$\bar{x}_\ell = x_\ell + \mu_\ell p^{\scal}_\ell$.
Then:
\medskip\\
(i) The inequality
\begin{equation}\label{e.fejer_uni}
\|x_{\ell+1}-x^*\|^2
\le
\|x_\ell-x^*\|^2
-
\Big(
{(x_\ell-\bar{x}_\ell)^T E(\bar{x}_\ell)}/
{\|E(\bar{x}_\ell)\|}
\Big)^2
\end{equation}
holds for every $x^*\in X^*$.
Consequently, the sequence $\{x_\ell\}_{\ell\geq0}$ is Fej\'er monotone with respect to $X^*$, i.e.,
\begin{equation}\label{e.Fejer}
\|x_{\ell+1}-x^*\|
\le
\|x_\ell-x^*\|,
\end{equation}
and the mentioned sequences are bounded.
\medskip\\
(ii) One of the following holds:
\medskip\\
(ii-1) Algorithm \ref{a.FDS} terminates in a finite number of iterations.
\medskip\\
(ii-2) Algorithm \ref{a.FDS} generates an infinite sequence $\{x_\ell\}_{\ell\geq 0}$ satisfying
\begin{equation}\label{e.sum2diffx}
\sum_{\ell=0}^{\infty}
\|x_\ell-\bar{x}_\ell\|^4
<\infty, \ \ \sum_{\ell=0}^{\infty}
\|x_{\ell+1}-x_\ell\|^2
<\infty
\end{equation}
and, respectively,  therefore
\begin{equation}\label{e.th0}
\lim_{\ell\to\infty}\|x_\ell-\bar{x}_\ell\|=0,
\qquad
\lim_{\ell\to\infty}\|x_{\ell+1}-x_\ell\|=0.
\end{equation}
\end{lem}

\begin{proof}
(i) Let $x^*\in X^*$ be arbitrary, so that $E(x^*)=0$. From \eqref{e.XH}, $x^*\in \ol H_\ell$, where $\ol H_\ell$ is defined in
\eqref{e:semspa}.
Since $x_{\ell+1}
=
P_{\ol H_\ell}(x_\ell)$ and $x^*\in X^*\subseteq \ol H_\ell$, 
\gzit{eq:fejer-descent} is satisfied. Moreover,
$\|x_{\ell+1}-x_\ell\|
=
\operatorname{dist}(x_\ell,\ol H_\ell)$. Because the sufficient descent inequality implies \gzit{eq:positive-inner-product}, we obtain $E(\bar{x}_\ell)\neq0$  and, the current point $x_\ell$ lies outside
$\ol H_\ell$. So \eqref{e.fejer_uni} is obtained from \eqref{e.lamb} and \eqref{eq:fejer-descent} as $\operatorname{dist}(x_\ell,\ol H_\ell)
={
(x_\ell-\bar{x}_\ell)^T E(\bar{x}_\ell)
}/{
\|E(\bar{x}_\ell)\|
}$. 
Therefore, the condition \eqref{e.Fejer} holds, so $\{x_\ell\}_{\ell\geq 0}$ is Fej\'er monotone with respect to $X^*$.
Hence, by induction $\|x_\ell-x^*\|
\le
\|x_0-x^*\|$ holds. Thus, $\{x_\ell\}_{\ell\geq 0}$ is bounded. From Lipschitz continuity of $E$,
\[
\|E(x_\ell)\|
=
\|E(x_\ell)-E(x^*)\|
\le
L\|x_\ell-x^*\|
\le
L\|x_0-x^*\|,
\]
so $\{E(x_\ell)\}_{\ell\geq 0}$ is bounded. By construction, $\{\mu_\ell\}_{\ell\geq 0}$ is bounded. Moreover, by \eqref{e:sccon}, $\|p_\ell^{\scal}\|
\le
\|E(x_\ell)\|$, 
and hence $\{p_\ell^{\scal}\}_{\ell\geq 0}$ is bounded. It follows directly
that $\{\bar{x}_\ell\}_{\ell\geq 0}$ is bounded from \eqref{e.xbar} with $p_\ell=p^{\scal}_\ell$. Since $E$ is
Lipschitz continuous,
the sequence $\{E(\bar{x}_\ell)\}_{\ell\geq 0}$ is also bounded.

(ii-1) If Algorithm \ref{a.FDS} terminates in finitely many iterations and $\|E(x_\ell)\|\le\eps$, then the last iterate $x_\ell$ is an $\eps$-approximate solution of
\eqref{e.ME}. In particular, if $E(x_\ell)=0$, then $x_\ell$ is a solution
of \eqref{e.ME}. If termination occurs solely because {\tt nRmax} is
reached, the algorithm stops because the prescribed residual-evaluation
budget has been exhausted, and no approximate-solution conclusion follows
unless the residual criterion is also satisfied.

(ii-2)
Since $x_\ell\neq\bar{x}_\ell$, at every nonstationary iteration, the sufficient descent inequality holds. In both cases of line search and fixed step sizes, the parameter $\rho>0$ is independent of $\ell$ and
satisfies
\begin{equation}\label{e.th2}
(x_\ell-\bar{x}_\ell)^T E(\bar{x}_\ell)
\ge
\rho \|x_\ell-\bar{x}_\ell\|^2.
\end{equation}
Indeed, in the line-search step-size case,  it was obtained from $x_\ell-\bar{x}_\ell
=
-\mu_\ell p_\ell^{\scal}$ and \gzit{e.rbls}, while it followed from Lemma~\ref{le:fixed-descent-modified} in the fixed step-size case,
provided that $\alpha$ satisfies \eqref{e.alpha_bound}. Applying \eqref{e.th2} to \eqref{e.fejer_uni} yields
\begin{equation}\label{e.th55}
\|x_{\ell+1}-x^*\|^2
\le
\|x_\ell-x^*\|^2
-
\rho^2
\frac{\|x_\ell-\bar{x}_\ell\|^4}
{\|E(\bar{x}_\ell)\|^2}.
\end{equation}
Summing the Fej\'er decrease inequality \eqref{e.th55} from $\ell=0$ to $N-1$ yields
\[
\sum_{\ell=0}^{N-1}
\rho^2
\frac{\|x_\ell-\bar{x}_\ell\|^4}
{\|E(\bar{x}_\ell)\|^2}
\le
\sum_{\ell=0}^{N-1}
\Big(
\|x_\ell-x^*\|^2
-
\|x_{\ell+1}-x^*\|^2
\Big)=
\|x_0-x^*\|^2
-
\|x_N-x^*\|^2,
\]
which gives
\begin{equation}\label{e.sum66}
\sum_{\ell=0}^{N-1}
\rho^2
\frac{\|x_\ell-\bar{x}_\ell\|^4}
{\|E(\bar{x}_\ell)\|^2}
\le
\|x_0-x^*\|^2
\end{equation}
since $\|x_N-x^*\|^2 \ge 0$. 
By part (i), 
$\{E(\bar{x}_\ell)\}_{\ell\geq 0}$ is bounded.
Hence, there exists a constant $C>0$ such that
\begin{equation}\label{e.bdC}
\|E(\bar{x}_\ell)\| \le C,
\qquad \forall \ell\ge0.
\end{equation}
From \eqref{e.bdC},  ${\|x_\ell-\bar{x}_\ell\|^4}/
{\|E(\bar{x}_\ell)\|^2}
\ge \|x_\ell-\bar{x}_\ell\|^4/C^2$
is obtained. Substituting this inequality into \eqref{e.sum66} and letting $N\to\infty$ show that the first condition in \eqref{e.sum2diffx} holds, namely, $\sum_{\ell=0}^{\infty}
\|x_\ell-\bar{x}_\ell\|^4
<\infty$.
Consequently, the first condition in \eqref{e.th0} holds, that is, $\|x_\ell-\bar{x}_\ell\|\to 0 $. Finally, using \eqref{eq:fejer-descent} and
summing this inequality from $\ell=0$ to $N-1$ yields
\[
\sum_{\ell=0}^{N-1}
\|x_{\ell+1}-x_\ell\|^2
\le
\|x_0-x^*\|^2
-
\|x_N-x^*\|^2
\le
\|x_0-x^*\|^2.
\]
Letting $N\to\infty$, the second condition in \eqref{e.sum2diffx} is obtained, namely $\sum_{\ell=0}^{\infty}
\|x_{\ell+1}-x_\ell\|^2
<\infty$, so that the second condition in \eqref{e.th0}, namely 
$
\|x_{\ell+1}-x_\ell\|\to 0
$, is obtained. \qed
\end{proof}
The following theorem shows that the sequence $\{x_\ell\}_{\ell\geq0}$
generated by Algorithm \ref{a.FDS} is globally convergent to a solution $x^*$ of the
system, i.e., $E(x^*)=0$.

\begin{thm}[Global convergence of Algorithm \ref{a.FDS}]\label{th:gl}
Let $\{x_\ell\}_{\ell\geq0}$ be an infinite sequence generated by  Algorithm \ref{a.FDS}
and suppose that {\rm (\Asol)} and {\rm (\Alip)} hold. Then
\begin{equation}\label{e.infz}
\liminf_{\ell\rightarrow\infty}\|E(x_\ell)\|=0.
\end{equation}
Moreover, the sequence $\{x_\ell\}_{\ell\geq0}$ converges to some
$x^*\in X^*$.
\end{thm}

\begin{proof}  
We argue by contradiction and assume that \eqref{e.infz} does not hold.
Then there exist constants $\eps>0$ and $\ell_0\ge0$ such that
\begin{equation}\label{contra_eps}
\|E(x_\ell)\|\ge\eps
\qquad
\text{for all }\ell\ge\ell_0.
\end{equation}

Under this assumption, the scaling condition \eqref{e:sccon} yields
\begin{equation}\label{e.plobd}
\|p_\ell^{\scal}\|
\ge
\zeta \|E(x_\ell)\|
\ge
\zeta\eps
=: c_0>0,
\end{equation}
for all $\ell\ge\ell_0$, where $c_0$ is independent of $\ell$. By Lemma~\ref{le:bo}(i), we have the Fej\'er monotonicity property \gzit{e.fejer_uni}, namely,
\[
\|x_{\ell+1}-x^*\|^2 
\le 
\|x_\ell-x^*\|^2 
- {\big((x_\ell-\bar x_\ell)^T E(\bar x_\ell)\big)^2}/
{\|E(\bar x_\ell)\|^2},
\qquad
x^*\in X^*.
\]

We now show that the numerator is bounded away from zero in both cases:

{\bf Line search step-size case.}  Since the accepted step size satisfies $\mu_\ell\ge\mu_{\min}>0$, using \gzit{e.rbls} and \eqref{e.plobd}, $(x_\ell-\bar x_\ell)^T E(\bar x_\ell)
\ge
\rho \mu_{\min}^2 \|p_\ell^{\scal}\|^2
\ge
\rho \mu_{\min}^2 c_0^2
=: \delta_{\rm var} > 0$.

{\bf Fixed step-size case.} 
By Lemma~\ref{le:fixed-descent-modified}, with
$\mu_\ell = \alpha$,  $(p_\ell^{\scal})^T E(\bar x_\ell)
\le
-\rho\alpha\|p_\ell^{\scal}\|^2$. Since $x_\ell-\bar x_\ell
=
-\alpha p_\ell^{\scal}$, from \eqref{e.rbls} and \eqref{e.plobd}, it follows that
\[
\begin{aligned}
(x_\ell-\bar x_\ell)^T E(\bar x_\ell)=
-\alpha
(p_\ell^{\scal})^T E(\bar x_\ell)
\ge
\rho\alpha^2\|p_\ell^{\scal}\|^2
\ge
\rho\alpha^2c_0^2
=: \delta_{\rm fix} > 0.
\end{aligned}
\]
Let $\delta:=\delta_{\rm var}$ in the line-search step-size case and
$\delta:=\delta_{\rm fix}$ in the fixed step-size case. In both cases, for all $\ell\ge\ell_0$, the numerator in \gzit{e.fejer_uni} is at least $\delta > 0$. 

By Lemma~\ref{le:bo}(i), the sequence
$\{E(\bar x_\ell)\}_{\ell\geq0}$ is bounded, i.e., there exists a constant
$C>0$ such that $\|E(\bar x_\ell)\| \le C$ for all $\ell\ge0$. Substituting these estimates into \gzit{e.fejer_uni} yields
$\|x_{\ell+1}-x^*\|^2 
\le 
\|x_\ell-x^*\|^2 
- {\delta^2}/{C^2}$. Let $N>\ell_0$ be an integer. For $\ell=\ell_0,\ldots,N-1$, recursively applying the preceding inequality yields
\[
0
\le
\|x_N-x^*\|^2
\le
\|x_{\ell_0}-x^*\|^2
-
(N-\ell_0){\delta^2}/{C^2},
\]
which leads to a contradiction as $N\to\infty$. Hence, assumption \gzit{contra_eps} is false and \eqref{e.infz} holds.

It remains to prove convergence of the entire sequence. By \eqref{e.infz}, there exists a subsequence
\(\{x_{\ell_j}\}_{j\geq0}\) such that $\|E(x_{\ell_j})\|\to0$. By Lemma~\ref{le:bo}(i), the sequence \(\{x_{\ell_j}\}_{j\geq0}\) is
bounded. Hence, by the
Bolzano-Weierstrass theorem, it admits a convergent subsequence.
Passing to this subsequence, which we continue to denote by
$\{x_{\ell_j}\}_{j\ge0}$, there exists \(\hat x\in\mathbb{R}^n\) such that $x_{\ell_j}\to\hat x$.

The Lipschitz continuity of $E$ implies its continuity, and hence $E(\hat x)
=
\lim_{j\to\infty}E(x_{\ell_j})
=
0$. Therefore, $\hat x\in X^*$. By the Fej\'er monotonicity property \eqref{e.Fejer}, the sequence
$\{\|x_\ell-\hat x\|\}_{\ell\geq0}$ is nonincreasing and therefore
convergent. Since $x_{\ell_j}\to\hat x$, we have $\lim_{j\to\infty}\|x_{\ell_j}-\hat x\|=0$. Consequently, $\lim_{\ell\to\infty}\|x_\ell-\hat x\|=0$, and hence $x_\ell\to\hat x\in X^*$. This completes the proof. \qed
\end{proof}

\section{Required Assumptions for Convergence Rates}

To analyze the convergence rates and iteration complexity of Algorithm \ref{a.FDS}, we impose a standard local error-bound assumption on the operator $E$. The rate analysis requires this property only in a neighborhood of the solution set.

\medskip
\noindent
(\Aeb) \textbf{Local Error Bound.}
There exist constants $\vartheta>0$ and $\delta>0$ such that, for any 
$x \in \mathcal N_\delta(X^*):= \{ x \in \Rz^n : \dist(x,X^*) < \delta \}$,
\begin{equation}\label{e.EB}
\vartheta\,\dist(x,X^*) \le \|E(x)\|.
\end{equation}
Assumption {\rm (\Aeb)} ensures that, locally around the solution set, the residual norm provides a reliable quantitative measure of the distance to $X^*$. 
This assumption is a metric-subregularity condition for the equation $E(x)=0$ and is satisfied by several important classes of nonlinear equations and variational problems under suitable regularity conditions.

(\Areg) \textbf{Local Jacobian Regularity.}
The operator $E$ is continuously differentiable on an open neighborhood
of $x^*\in X^*$, and $\sigma_{\min}(J_E(x^*))>0$, where $\sigma_{\min}(\cdot)$ denotes the smallest singular value.

\textbf{Structural sufficient condition for (\Aeb).}
The local error bound is a structural property of the operator and does
not depend on the algorithm. Assumption {\rm (\Areg)} is a classical
sufficient condition for a local error bound around the particular
solution $x^*$. Indeed, the nonsingularity of $J_E(x^*)$, together with
the continuous differentiability of $E$, implies that $x^*$ is locally
isolated and that there exist constants $\cc>0$ and $\delta_*>0$ such that $\|x-x^*\|
\le
\cc\|E(x)\|$ for all $x\in B_{\delta_*}(x^*)$. Here, $B_\delta(x^*)$ denotes the open ball centered at a particular solution
$x^*\in X^*$, whereas $\mathcal N_\delta(X^*)$ denotes a neighborhood of
the entire solution set $X^*$. Since $x^*$ is the unique solution in a sufficiently small neighborhood, we have $\dist(x,X^*)=\|x-x^*\|$ for all $x$ sufficiently close to $x^*$. Therefore, after possibly
reducing the neighborhood and setting $\vartheta:=1/\cc$, the estimate $\vartheta\,\dist(x,X^*)
\le
\|E(x)\|$ 
holds locally around $x^*$. This pointwise sufficient condition does not, by itself, imply that the
same constants $\vartheta$ and $\delta$ are valid uniformly on a neighborhood of the entire solution set $X^*$.

\begin{thm}[Sufficient condition for a local error bound]
\label{thm:structural-eb}
Suppose that $E$ is continuously differentiable on an open neighborhood
of $x^*\in X^*$ and that the Jacobian $J_E(x^*)$ is nonsingular. Then
there exist constants $\vartheta>0$ and $\delta>0$ such that \begin{equation}\label{e.localEB}
\vartheta\,\dist(x,X^*)
\le
\|E(x)\|,
\qquad
\forall x\in B_\delta(x^*).
\end{equation}
\end{thm}

\begin{proof}
Since $E$ is continuously differentiable near $x^*\in X^*$, we write
\[
E(x) = E(x^*) + J_E(x^*)(x-x^*) + h(x)
= J_E(x^*)(x-x^*) + h(x),
\]
where
\begin{equation}\label{e.res}
\|h(x)\|
=
o(\|x-x^*\|)
\qquad
\text{as }x\to x^*.
\end{equation}
Let $c_0 := \sigma_{\min}(J_E(x^*))>0$ denote the smallest singular value. 
The definition of the smallest singular value of a matrix, $\sigma_{\min}(A)=\min\{\|Av\|\mid \|v\|=1\}$, implies that $\|Av\|\ge \sigma_{\min}(A)\|v\|$ for all $v$. Applying this with $A=J_E(x^*)$ and $v=x-x^*$ yields $\|J_E(x^*)(x-x^*)\|
\ge c_0 \|x-x^*\|$. By \eqref{e.res}, there exists a neighborhood of $x^*$ such that for all $x$ in it, $\|h(x)\| \le ({c_0}/{2})\|x-x^*\|$. Therefore,
\[
\|E(x)\| 
\ge \|J_E(x^*)(x-x^*)\|-\|h(x)\|
\ge c_0 \|x-x^*\| - \frac{c_0}{2}\|x-x^*\|
= \frac{c_0}{2}\|x-x^*\|.
\]
By the Inverse Function Theorem, the nonsingularity of $J_E(x^*)$
implies that $E$ is locally invertible around $x^*$. So, $x^*$ is an isolated (or locally unique) solution of \eqref{e.ME}.

Choose \(r>0\) such that \(X^*\cap B_r(x^*)=\{x^*\}\). Then, for every
\(x\in B_{r/3}(x^*)\) and \(z\in X^*\setminus\{x^*\}\),
\[
\|x-z\|
\ge
\|z-x^*\|-\|x-x^*\|
>
r-r/3={2r}/{3}
>
\|x-x^*\|.
\]
Hence, \(x^*\) is the unique closest point in \(X^*\) to \(x\), and so $\dist(x,X^*)=\|x-x^*\|$.  Thus, after reducing the neighborhood if necessary,  setting $\vartheta:=c_0/2$
proves the local error bound \eqref{e.localEB}, i.e., $\|E(x)\|
\ge
\vartheta\|x-x^*\|
=
\vartheta\dist(x,X^*)$. \qed
\end{proof}

\paragraph{\bfi{Relation to Strong Monotonicity and Jacobian Regularity.}}

Let $E$ be continuously differentiable and monotone on an open neighborhood of $x^*$. Then the symmetric part of the Jacobian at any point $x^*$ satisfies ${(J_E(x^*) + J_E(x^*)^\top)}/{2} \succeq 0$. Indeed, using the monotonicity condition \eqref{e.monotone}, fix
$v\in\mathbb{R}^n$, set $x=x^*+tv$ and $y=x^*$, then $\langle E(x^*+tv)-E(x^*),tv\rangle \ge 0$. For every $t\neq0$, dividing by $t^2$ gives $\left\langle{(E(x^*+tv)-E(x^*))}/{t},
v
\right\rangle
\ge0$. Since $E$ is differentiable at $x^*$,
$t^{-1}(E(x^*+tv)-E(x^*))
\to
J_E(x^*)v$ as  $t\to0$. Passing to the limit yields $\langle J_E(x^*)v,v\rangle \ge 0$ for all $v\in\mathbb{R}^n$. Since the skew--symmetric part ${(A-A^\top)}/{2}$ of a matrix $A$
satisfies $v^\top(A-A^\top)v=0$ for all $v$, we have $\langle Av,v\rangle
=
\left\langle
{(A+A^\top)v}/{2},
v
\right\rangle$. Applying this identity with $A=J_E(x^*)$ gives
$\left\langle
{(J_E(x^*) + J_E(x^*)^\top)v}/{2},
v
\right\rangle
=
\langle J_E(x^*)v,v\rangle
\ge 0$ for all $v$. This is precisely the positive semidefiniteness of the symmetric part of
$J_E(x^*)$. If, in addition, ${(J_E(x^*) + J_E(x^*)^\top})/{2} \succ 0$, since the smallest eigenvalue is positive, $\mu_*>0$, so $\left\langle
{(J_E(x^*) + J_E(x^*)^\top)v}/{2},
v
\right\rangle
\ge
\mu_*\|v\|^2$ for all $v$. Using the identity above, this is equivalent to $\langle J_E(x^*)v,v\rangle
\ge
\mu_*\|v\|^2$ for all $v$. Moreover, by continuity of $J_E$, there exist constants $\mu>0$ and
$\delta>0$ such that ${(J_E(x)+J_E(x)^\top)}/{2}
\succeq
\mu I$ for all $x\in B_\delta(x^*)$. Since $B_\delta(x^*)$ is convex, the segment joining any
$x,y\in B_\delta(x^*)$ is contained in $B_\delta(x^*)$, and the mean-value integral formula gives
\[
\langle E(x)-E(y),x-y\rangle
\ge
\mu\|x-y\|^2,
\qquad
\forall x,y\in B_\delta(x^*).
\]
Consequently, $E$ is strongly monotone on this neighborhood. This pointwise Jacobian condition is weaker than global strong
monotonicity, which requires the uniform inequality $\langle E(x)-E(y), x-y\rangle 
\ge \mu \|x-y\|^2$ for all $x,y \in \Rz^n$ and some constant $\mu>0$. The relationship between these concepts is classical in monotone
operator theory; see, for instance,
\cite[Proposition~22.9]{Bauschke2017}.

\section{Convergence rate and complexity}\label{s.convRate}

This section investigates the convergence rates and complexity of
Algorithm \ref{a.FDS}. We first establish auxiliary estimates connecting the
projection-step length to the residual norm. These estimates are then used
to derive a global best-iterate residual rate, a sharper asymptotic rate for
the distance to the solution set under local regularity, and local
\(R\)-linear convergence under Jacobian regularity. Finally, we derive the
corresponding iteration and residual-evaluation complexity bounds.

\subsection{Auxiliary results for convergence-rate analysis}

While the unscaled subspace directions $p_\ell$ generated by the CG-type recursion \eqref{e.pcg1} are not required to be bounded, the algorithm employs the scaled directions $p_\ell^{\mathrm{scal}}$, defined by \eqref{e.scaleDir}, satisfying the scaling condition \eqref{e:sccon}. 
This scaling is essential for ensuring both global convergence and the validity of the convergence-rate analysis.

\begin{lem}[Lower bound on the trial step]\label{le:lowerstep}
Suppose that Assumptions {\rm (\Asol)}--{\rm (\Alip)} hold and that Algorithm \ref{a.FDS} generates an infinite sequence
$\{x_\ell\}_{\ell\ge0}$. Let the trial point $\bar x_\ell$ be defined by
\eqref{e.xbar}, that is, $\bar x_\ell
=
x_\ell+\mu_\ell p_\ell^{\scal}$, where $p_\ell^{\scal}$ is the rescaled direction satisfying the enforced
angle condition \eqref{e.rac} and the norm condition \eqref{e:sccon}. Let
$\mu_\ell>0$ be the step size obtained from the line search
\eqref{e.bls}. Given the known Lipschitz constant $L>0$ and the constants $\Delta_{\angle},\rho\in(0,1)$ and $\gamma>1$, we choose the
algorithmic safeguard parameter as \eqref{e.mu-min}, that is $\mu_{\min}
:=
\min\left\{
\mu_{\init},{\Delta_{\angle}}
/{\gamma(L+\rho)}
\right\}$. Then there exists a constant $\kappa>0$ such that
\begin{equation}\label{e.lowerstep}
\|x_\ell-\bar x_\ell\|
\ge
\kappa\,\|E(x_\ell)\|,
\qquad
\text{for all }\ell\ge0.
\end{equation}
\end{lem}

\begin{proof}
By Lemma \ref{lem:lipschitz-direction}, $(p_\ell^{\scal})^T E(\bar x_\ell)
\le
-(\Delta_{\angle}-L\mu)
\|p_\ell^{\scal}\|^2$. Hence, the line-search condition \eqref{e.bls} is satisfied whenever
\[
-(\Delta_{\angle}-L\mu)
\|p_\ell^{\scal}\|^2
\le
-\rho\mu\|p_\ell^{\scal}\|^2.
\]
Equivalently, it is sufficient that $\mu
\le {\Delta_{\angle}}/{(L+\rho)}
=:
\bar\mu$. Since $L$ is known and $\gamma>1$, we choose the minimal allowable step size
algorithmically as \eqref{e.mu-min}. Hence, the backtracking procedure is well-defined, and every accepted
step size satisfies $\mu_\ell
\ge
\mu_{\min}$, for all $\ell\ge0$. Using the lower scaling condition in \eqref{e:sccon}, we obtain $\|x_\ell-\bar x_\ell\|
=
\mu_\ell\|p_\ell^{\scal}\|
\ge
\mu_{\min}\zeta\|E(x_\ell)\|$. Setting $\kappa
:=
\mu_{\min}\zeta>0$ completes the proof. \qed
\end{proof}

\begin{cor}[Sufficient descent of the projection step]
\label{cor:step-lower-strong}
Suppose that Assumptions {\rm (\Asol)}--{\rm (\Alip)} hold and that
Algorithm \ref{a.FDS} generates an infinite sequence using either the line-search step-size implementation covered by Lemma~\ref{le:lowerstep} or the fixed step-size implementation with $\mu_\ell=\alpha$ satisfying
\eqref{e.alpha_bound}. Then, the projection update satisfies
\begin{equation}\label{eq:step-lower-strong}
\|x_{\ell+1}-x_\ell\|
\ge
c\,\|E(x_\ell)\|,
\end{equation}
where $c>0$ is independent of $\ell$. More precisely,
\[
c
=
\begin{cases}
\displaystyle
\frac{
\rho\mu_{\min}^2\zeta^2
}{
1+L\mu_{\init}
},
&
\text{for the line-search step-size case},
\\[4mm]
\displaystyle
\frac{
\rho\alpha^2\zeta^2
}{
1+L\alpha
},
&
\text{for the fixed step-size case}.
\end{cases}
\]
\end{cor}

\begin{proof}
Since the algorithm generates an infinite sequence, the stopping criterion
is not satisfied at any iteration. In particular, $E(x_\ell)\neq0$,
for all $\ell\ge0$. Since $x_{\ell+1}
=
P_{\ol H_\ell}(x_\ell)$, using \eqref{e.proj}, we obtain
\begin{equation}\label{e.step-lower}
\|x_{\ell+1}-x_\ell\|
=
\dist(x_\ell,\ol H_\ell)
=
{
(x_\ell-\bar x_\ell)^T E(\bar x_\ell)
}/{
\|E(\bar x_\ell)\|
}.
\end{equation}
Moreover, since $\mu_\ell\le\mu_{\init}$ in the line-search step-size implementation and $\mu_\ell=\alpha$ in the fixed step-size implementation, the Lipschitz continuity of $E$ and the upper scaling condition in \eqref{e:sccon} yield
\begin{equation}\label{e.residual-trial-bound}
\begin{aligned}
\|E(\bar x_\ell)\|
&\le
\|E(x_\ell)\|
+
\|E(\bar x_\ell)-E(x_\ell)\|
\le
\|E(x_\ell)\|
+
L\mu_\ell\|p_\ell^{\scal}\|\\
&\le
C\|E(x_\ell)\|.
\end{aligned}
\end{equation}
where $C
:=
1+L\mu_{\init}
>0$ in variant and $C
:=
1+L\alpha
>0$ in fixed case, by Step~({\sc S1}) of
Algorithm \ref{a.FDS}. We consider the two step-size implementations separately.

\medskip
\noindent
\textbf{Line search step-size implementation.}
Substituting \gzit{e.rbls} into \gzit{e.step-lower}, we obtain
\begin{equation}\label{e.step-lower-descent}
\|x_{\ell+1}-x_\ell\|
\ge
\rho
{
\mu_\ell^2\|p_\ell^{\scal}\|^2
}/{
\|E(\bar x_\ell)\|
}.
\end{equation}
Furthermore, by \eqref{e:sccon} and \eqref{e.mu-min}, we have $\|p_\ell^{\scal}\|
\ge
\zeta\|E(x_\ell)\|$ and $\mu_\ell
\ge
\mu_{\min}
$, respectively. Combining these inequalities along with \gzit{e.residual-trial-bound} in \gzit{e.step-lower-descent} yields
\[
\begin{aligned}
\|x_{\ell+1}-x_\ell\|
&\ge
\rho
\frac{
\mu_\ell^2\|p_\ell^{\scal}\|^2
}{
\|E(\bar x_\ell)\|
}
\ge
\rho
\frac{
\mu_{\min}^2\zeta^2\|E(x_\ell)\|^2
}{
C\|E(x_\ell)\|
} =
\frac{
\rho\mu_{\min}^2\zeta^2
}{
1+L\mu_{\init}
}
\|E(x_\ell)\|.
\end{aligned}
\]

\medskip
\noindent
\textbf{Fixed step-size implementation.}
When Algorithm \ref{a.FDS} is implemented with a fixed step size
$\mu_\ell=\alpha$ satisfying \eqref{e.alpha_bound}, the distance between
the current iterate and the trial point satisfies
$
\|x_\ell-\bar x_\ell\|
=
\alpha\|p_\ell^{\scal}\|
\ge
\alpha\zeta\|E(x_\ell)\|,
$
where the last inequality follows from the lower scaling condition in
\eqref{e:sccon}. Moreover, Lemma~\ref{lem:lipschitz-direction} yields
$
(p_\ell^{\scal})^T E(\bar x_\ell)
\le
-\rho\alpha\|p_\ell^{\scal}\|^2.
$
Hence, by \eqref{e.rbls},
$
(x_\ell-\bar x_\ell)^T E(\bar x_\ell)
\ge
\rho\alpha^2\|p_\ell^{\scal}\|^2.
$
Therefore, \eqref{e.step-lower} gives
\[
\|x_{\ell+1}-x_\ell\|
\ge
\rho
{
\alpha^2\|p_\ell^{\scal}\|^2
}/{
\|E(\bar x_\ell)\|
}.
\]
Using the residual bound \gzit{e.residual-trial-bound} together with the
lower scaling condition in \eqref{e:sccon}, we obtain
\[
\begin{aligned}
\|x_{\ell+1}-x_\ell\|
&\ge
\rho
\frac{
\alpha^2\|p_\ell^{\scal}\|^2
}{
\|E(\bar x_\ell)\|
}
\ge
\rho
\frac{
\alpha^2\zeta^2\|E(x_\ell)\|^2
}{
C\|E(x_\ell)\|
}
\ge
\frac{
\rho\alpha^2\zeta^2
}{
1+L\alpha
}
\|E(x_\ell)\|.
\end{aligned}
\]
Therefore, in both implementations, \eqref{eq:step-lower-strong} holds
with the stated constant $c>0$ independent of $\ell$. \qed
\end{proof}

The following result is similar to
\cite[Proposition~5.4(iii)]{Bauschke2017}.
\begin{lem}\label{le:distzero}
Suppose that Assumptions {\rm (\Asol)} and {\rm (\Alip)} hold and that
the sequence $\{x_\ell\}_{\ell\ge0}$ generated by Algorithm \ref{a.FDS} is
infinite. Then
\begin{equation}\label{e.distzero}
\lim_{\ell\to\infty}
\dist(x_\ell,X^*)
=
0.
\end{equation}
\end{lem}

\begin{proof}
By Theorem~\ref{th:gl}, there exists $x^*\in X^*$ such that $x_\ell
\to
x^*$. Since the distance function to the nonempty closed set $X^*$ is
continuous, it follows that $\dist(x_\ell,X^*)$ converges to $\dist(x^*,X^*)
=
0$. This completes the proof. \qed
\end{proof}

\begin{lem}[Discrete comparison principles]
\label{lem:disc-comparison}
Let $\{a_\ell\}_{\ell\ge0}$ be a nonnegative sequence satisfying $a_{\ell+1}
\le
a_\ell$ for all $\ell\ge0$.
\begin{enumerate}
\item[(i)]
Suppose that there exists a sequence of positive constants
$\{\gamma_\ell\}_{\ell\ge\ell_0}$ such that
\begin{equation}\label{eq:quartic-rec-gen}
a_{\ell+1}^2
\le
a_\ell^2
-
\gamma_\ell a_\ell^4,
\qquad
\forall \ell\ge\ell_0.
\end{equation}
If $\liminf_{\ell\to\infty}\gamma_\ell
>
0$, then $a_\ell
=
O(\ell^{-1/2})$.

\item[(ii)]
Independently of part {\rm (i)}, suppose that $\sum_{\ell=0}^{\infty}a_\ell^2
<
\infty$. Then, $\ell a_\ell^2
\to
0$, and hence $
a_\ell
=
o(\ell^{-1/2})$.
\end{enumerate}
\end{lem}

\begin{proof}
{\rm (i)}
If $a_\ell=0$ for some $\ell\ge\ell_0$, then the nonnegativity and
monotonicity of $\{a_\ell\}_{\ell\ge0}$ imply that $a_k=0$,  for all $k\ge\ell$, and the conclusion follows immediately. Assume, therefore, that $a_\ell>0$, for all $\ell\ge\ell_0$. Since $a_{\ell+1}
\le
a_\ell$, we have ${a_\ell^2}/{a_{\ell+1}^2}
\ge
1$. Dividing \eqref{eq:quartic-rec-gen} by
$a_\ell^2a_{\ell+1}^2$ gives
\[
{1}/{a_{\ell+1}^2}
- {1}/{a_\ell^2}
\ge
\gamma_\ell
{a_\ell^2}/{a_{\ell+1}^2}
\ge
\gamma_\ell,
\qquad
\forall \ell\ge\ell_0.
\]
Summing from $k=\ell_0$ to $\ell-1$ yields ${1}/{a_\ell^2}
\ge{1}/{a_{\ell_0}^2}
+
\sum_{k=\ell_0}^{\ell-1}\gamma_k$.
Since $\liminf_{\ell\to\infty}\gamma_\ell
>
0$, there exist constants $\gamma>0$ and $\ell_1\ge\ell_0$ such that $\gamma_\ell
\ge
\gamma$ for all $\ell\ge\ell_1$. Hence, for every $\ell>\ell_1$, $\sum_{k=\ell_1}^{\ell-1}\gamma_k
\ge
\gamma(\ell-\ell_1)$. Therefore, ${1}/{a_\ell^2}
\ge
\gamma(\ell-\ell_1)$,  and consequently $a_\ell
\le{1}/{
\sqrt{\gamma(\ell-\ell_1)}
}$. Thus, $a_\ell
=
O(\ell^{-1/2})$.

\medskip
\noindent
{\rm (ii)}
For each sufficiently large $\ell$, define $m_\ell
:=
\left\lfloor
{\ell}/{2}
\right\rfloor$. Since $\{a_\ell\}_{\ell\ge0}$ is nonincreasing, 
$(\ell-m_\ell+1)a_\ell^2
\le
\sum_{k=m_\ell}^{\ell}a_k^2$.
It follows that $\ell a_\ell^2
\le({\ell}/{
(\ell-m_\ell+1))
}
\sum_{k=m_\ell}^{\ell}a_k^2
\le
2
\sum_{k=m_\ell}^{\infty}a_k^2$. Since $m_\ell
\to
\infty$
and $\sum_{k=0}^{\infty}a_k^2
<
\infty$,  the right-hand side converges to zero. Hence, $\ell a_\ell^2
\to
0$. Equivalently, $\sqrt{\ell}\,a_\ell
\to
0$, and therefore $a_\ell
=
o(\ell^{-1/2})$. \qed
\end{proof}

\subsection{Baseline Sublinear Rates under Exact Projection}
\label{subsec:baseline_sublinear}

In this subsection, we establish a baseline sublinear convergence rate for
Algorithm \ref{a.FDS} under the exact projection rule. Under Assumptions
{\rm (\Asol)} and {\rm (\Alip)}, we derive a global best-iterate residual
rate using Fej\'er monotonicity, square summability of the projection
steps, and the lower bound relating the projection-step length to the
residual norm. No strong monotonicity assumption is required.

\begin{thm}[Baseline convergence rate of Algorithm \ref{a.FDS}]
\label{thm:sublinear_rate}
Let Assumptions {\rm (\Asol)} and {\rm (\Alip)} hold, and suppose that
Algorithm \ref{a.FDS} uses the exact projection rule. Then 
\begin{equation}\label{e.sumResidualSquares}
\sum_{\ell=0}^{\infty}
\|E(x_\ell)\|^2
<
\infty
\end{equation}
and therefore there exists a constant $C>0$ such that
\begin{equation}\label{e.mininq}
(\ell+1)
\min_{0\le k\le\ell}
\|E(x_k)\|^2
\le
\sum_{k=0}^{\ell}
\|E(x_k)\|^2
\le
C;
\end{equation} 
i.e., the best-iterate estimate $\D\min_{0\le k\le \ell}\|E(x_k)\|
=
O(\ell^{-1/2})$ holds.
\end{thm}

\begin{proof}
The asymptotic rate statement concerns an infinite sequence. If the
algorithm terminates finitely with \(E(x_{\ell_*})=0\), the conclusion
is immediate after extending the sequence by
\(x_\ell:=x_{\ell_*}\) for all \(\ell\ge\ell_*\). We therefore assume that the sequence
$\{x_\ell\}_{\ell\ge0}$ is infinite.

By Corollary~\ref{cor:step-lower-strong}, there exists a constant $c>0$,
independent of $\ell$, such that $\|x_{\ell+1}-x_\ell\|
\ge
c\|E(x_\ell)\|$. Consequently, $\|E(x_\ell)\|^2
\le({1}/{c^2})
\|x_{\ell+1}-x_\ell\|^2$. Combining this inequality with \eqref{e.sum2diffx} gives \eqref{e.sumResidualSquares}, i.e.,
\[
\sum_{\ell=0}^{\infty}
\|E(x_\ell)\|^2
\le
\frac{1}{c^2}
\sum_{\ell=0}^{\infty}
\|x_{\ell+1}-x_\ell\|^2
<
\infty.
\]
Hence, there exists a constant $C>0$ such that \eqref{e.mininq} holds. Therefore, $\D\min_{0\le k\le\ell}
\|E(x_k)\|
\le
\sqrt{{C}/{(\ell+1)}}
=
O(\ell^{-1/2})$. \qed
\end{proof}

\subsection{Linear convergence and $o(\ell^{-1})$ residual decay}\label{subsec:JFS_inertial_rate}

In this subsection, we first establish the asymptotic distance rate
\(o(\ell^{-1/2})\) under the local error-bound Assumption {\rm (\Aeb)}.
We then establish local linear convergence of Algorithm \ref{a.FDS}
under the exact projection rule and the stronger Assumption {\rm (\Areg)}. By
Theorem~\ref{thm:structural-eb}, this structural condition induces a local
error bound in a neighborhood of the limit solution and locally isolates that
solution. Once the iterates enter this neighborhood, the error bound combines
with the Fej\'er decrease and the uniform lower bound on the projection-step
length to yield a uniform contraction of the distance to the solution set.
The local isolation supplied by Assumption {\rm (\Areg)} then identifies this
distance with the distance to the limit solution, yielding \(R\)-linear
convergence of the iterates. The residual sequence consequently converges
geometrically, implying the sharper last-iterate decay
\(o(\ell^{-1})\). The argument does not require global strong monotonicity
or global uniqueness of the solution set.

\begin{thm}[Linear convergence and $o(\ell^{-1})$ residual decay]
\label{thm:JFS_inertial_rate}
Let $\{x_\ell\}_{\ell\geq0}$ be an infinite sequence generated by Algorithm \ref{a.FDS} using the exact
projection rule, and suppose that Assumptions {\rm (\Asol)} and {\rm (\Alip)} hold.
If Assumption~{\rm(\Aeb)} holds, then the distance sequence $d_\ell
=
\dist(x_\ell,X^*)$ in \eqref{e.defdell} satisfies $d_\ell
=
o(\ell^{-1/2})$. Moreover, if Assumption~{\rm(\Areg)} holds at the limit solution \(x^*\) supplied by
Theorem~\ref{th:gl}, then there exist an index $\ell_0\ge0$, a constant $q\in(0,1)$, and a constant $M>0$ such
that
\begin{equation}\label{e.linearsys}
\|x_\ell-x^*\|
\le
Mq^{\,\ell-\ell_0}, 
\qquad \|E(x_\ell)\|
\le
LMq^{\,\ell-\ell_0}
\qquad \forall \ell\ge\ell_0;
\end{equation}
i.e., $\{x_\ell\}_{\ell\geq0}$ converges
$R$-linearly to a solution $x^*\in X^*$ and the residual sequence
satisfies $\|E(x_\ell)\|
=
o(\ell^{-1})$.
\end{thm}

\begin{proof}
The proof has three steps.

\textbf{Step 1: Local error bound and sublinear distance rate.}
By Theorem~\ref{th:gl}, there exists $x^*\in X^*$ such that $x_\ell
\to
x^*$. Moreover, by Lemma~\ref{le:distzero}, $d_\ell
=
\dist(x_\ell,X^*)
\to
0$.

Suppose first that Assumption {\rm (\Aeb)} holds. Then, for all sufficiently
large $\ell$, we have $x_\ell\in\mathcal N_\delta(X^*)$, and hence, there exists $\ell_0\ge0$ such that
\begin{equation}\label{e.effbou}
\|E(x_\ell)\|
\ge
\vartheta d_\ell,
\qquad
\forall \ell\ge\ell_0.
\end{equation}

If instead Assumption {\rm (\Areg)} holds at the limit solution $x^*$, then
Theorem~\ref{thm:structural-eb} gives constants $\delta>0$ and $\vartheta>0$
such that \gzit{e.localEB} holds. Since $x_\ell\to x^*$, after increasing
$\ell_0$ if necessary, $x_\ell\in B_\delta(x^*)$ for all
$\ell\ge\ell_0$, and therefore \eqref{e.effbou} holds in this case as well. Equivalently,
\begin{equation}\label{eq:step3_triangle_clean}
d_\ell^2
\le
\frac{1}{\vartheta^2}
\|E(x_\ell)\|^2,
\qquad
\forall \ell\ge\ell_0.
\end{equation}
From \eqref{e.sumResidualSquares}, $\sum_{\ell=0}^{\infty}
\|E(x_\ell)\|^2
<
\infty$. It follows that
\begin{equation}\label{e.tauxZe}
\sum_{\ell=\ell_0}^{\infty}
d_\ell^2
\le
\frac{1}{\vartheta^2}
\sum_{\ell=\ell_0}^{\infty}
\|E(x_\ell)\|^2
<
\infty.
\end{equation}
Adding the finite initial part yields
\begin{equation}\label{eq:summ_d2_final}
\sum_{\ell=0}^{\infty}
d_\ell^2
<
\infty.
\end{equation}
From \eqref{eq:fejer-descent}, $d_{\ell+1}
\le
d_\ell$ for all $\ell\ge0$. Therefore, by \eqref{eq:summ_d2_final},
Lemma~\ref{lem:disc-comparison}(ii) yields $\ell d_\ell^2
\to
0$. Equivalently, $\sqrt{\ell}\,d_\ell
\to
0$, and hence
$d_\ell
=
o(\ell^{-1/2})$.

For the remainder of the proof, suppose that Assumption {\rm (\Areg)}
holds. By Theorem~\ref{thm:structural-eb}, after reducing $\delta$ if
necessary, $x^*$ is the unique solution of $E(x)=0$ in
$B_\delta(x^*)$. Since $x_\ell\to x^*$, after increasing $\ell_0$ if
necessary, we have
$x_\ell
\in
B_{\delta/3}(x^*)$ for all $\ell\ge\ell_0$. For every $\ell\ge\ell_0$, any solution outside $B_\delta(x^*)$ is at
distance greater than $2\delta/3$ from $x_\ell$, whereas $\|x_\ell-x^*\|
<{\delta}/{3}$. Consequently from \gzit{e.defdell}, $d_\ell
=
\dist(x_\ell,X^*)=
\|x_\ell-x^*\|$, for all $\ell\ge\ell_0$.

\medskip
\noindent
\textbf{Step 2: $R$-linear convergence of $\{x_\ell\}_{\ell\ge0}$.}
By \eqref{eq:fejer-descent}, $d_{\ell+1}^2
\le
d_\ell^2
-
\|x_{\ell+1}-x_\ell\|^2$. Combining \eqref{eq:step-lower-strong} with \eqref{e.effbou}, we obtain $\|x_{\ell+1}-x_\ell\|
\ge
c\vartheta d_\ell$ for all $\ell\ge\ell_0$. Substituting this bound into \eqref{eq:fejer-descent} gives $d_{\ell+1}^2
\le
\bigl(1-c^2\vartheta^2\bigr)d_\ell^2$ 
for all $\ell\ge\ell_0$. If $c^2\vartheta^2\ge1$, then finite convergence follows. Since the generated sequence is infinite, this case is excluded. Therefore,
$0<c^2\vartheta^2<1$, and we define
$q
:=
\sqrt{1-c^2\vartheta^2}
\in
(0,1)$. Then, $d_{\ell+1}
\le
q\,d_\ell$  for all $\ell\ge\ell_0$. Iterating this inequality yields
$d_\ell
\le
d_{\ell_0}q^{\,\ell-\ell_0}$ for all $\ell\ge\ell_0$. Since
$d_\ell=\|x_\ell-x^*\|$ for all $\ell\ge\ell_0$, we obtain $\|x_\ell-x^*\|
\le
Mq^{\,\ell-\ell_0}$ for all $\ell\ge\ell_0$, where $M:=d_{\ell_0}$. Therefore, $\{x_\ell\}_{\ell\ge0}$ converges
$R$-linearly to $x^*$.

\textbf{Step 3: Residual decay.}
Since $E(x^*)=0$ and {\rm (\Alip)} holds, we obtain
\[
\|E(x_\ell)\|
=
\|E(x_\ell)-E(x^*)\|
\le
L\|x_\ell-x^*\|.
\]
Therefore, for every $\ell\ge\ell_0$, $\|E(x_\ell)\|
\le
LMq^{\,\ell-\ell_0}
=
Cq^\ell$, where $C
:=
LMq^{-\ell_0}$.  Since $\ell q^\ell
\to0$ as $\ell\to\infty$, it follows that $\ell\|E(x_\ell)\|
\le
C\ell q^\ell
\to
0$. Hence, $\|E(x_\ell)\|
=
o(\ell^{-1})$. \qed
\end{proof}

\begin{rem}[Role of the local error bound]
The linear contraction of the distance sequence does not rely on strong
monotonicity of $E$. Once the iterates enter a neighborhood in which a local
error bound holds, the decrease induced by the exact projection step combines
with the error bound to produce a uniform contraction of the distance
sequence. The contraction factor is determined by the projection-step
constant $c$ and the local error-bound constant $\vartheta$. Under Assumption
{\rm (\Areg)}, the additional local isolation of the limit solution identifies
$d_\ell$ with $\|x_\ell-x^*\|$ for all sufficiently large $\ell$, and hence
the distance contraction yields $R$-linear convergence of the iterates to
$x^*$.
\end{rem}

\subsection{Complexity Analysis}

In this subsection, we derive iteration and residual-evaluation complexity
bounds for Algorithm \ref{a.FDS} under the exact projection rule. The global
best-iterate residual estimate obtained in
Theorem~\ref{thm:sublinear_rate} yields an
\(O(\varepsilon^{-2})\) iteration bound for reaching a residual of size
at most \(\varepsilon\), while the local \(R\)-linear convergence result
of Theorem~\ref{thm:JFS_inertial_rate} yields the pointwise bound
\(O(\log\varepsilon^{-1})\). For the line-search step-size implementation, we
additionally bound the number of residual evaluations performed by the
backtracking line search using the uniform safeguard
\(\mu_{\min}>0\). These bounds count residual evaluations but do not include
the application-dependent cost of evaluating the operator itself.

\begin{thm}[Iteration and residual-evaluation complexity]
\label{thm:complexity}
Suppose that Assumptions {\rm (\Asol)} and {\rm (\Alip)} hold, and let
$\{x_\ell\}_{\ell\ge0}$ be the infinite sequence generated by
Algorithm \ref{a.FDS} using the exact projection rule.

\begin{enumerate}

\item[\rm (i)]
\textbf{Baseline best-iterate complexity.}
For every target accuracy $\varepsilon>0$, there exists an index
$k_\varepsilon\ge0$ such that $\|E(x_{k_\varepsilon})\|
\le
\varepsilon$ and
\begin{equation}\label{eq:iteration_bound_sub}
k_\varepsilon
\le
\left\lceil
C_{\rm sub}\varepsilon^{-2}
\right\rceil
-1
=
O(\varepsilon^{-2}),
\end{equation}
where $C_{\rm sub}>0$ is independent of $\varepsilon$.

\item[\rm (ii)]
\textbf{Local linear-regime complexity.}
Suppose, in addition, that Assumption {\rm (\Areg)} holds at the limit
point $x^*\in X^*$. Then there exist an index $\ell_0\ge0$ and constants
$q\in(0,1)$ and $C_{\rm lin}>0$ such that the first index
\[
\ell_\varepsilon
:=
\min
\left\{
\ell\ge\ell_0:
\|E(x_\ell)\|
\le
\varepsilon
\right\}
\]
satisfies
\begin{equation}\label{eq:iteration_bound_lin}
\ell_\varepsilon
\le
\ell_0+
\max
\left\{
0,
\left\lceil
{
\log(C_{\rm lin}/\varepsilon)
}/{
\log(1/q)
}
\right\rceil
\right\}
=
O\!\left(
\log\varepsilon^{-1}
\right).
\end{equation}

\item[\rm (iii)]
\textbf{Residual-evaluation complexity.}
For the line-search step-size implementation, suppose that the backtracking
procedure uses a factor $\gamma>1$, starts from $\mu_{\init}>0$, and
generates trial step sizes by the safeguarded update
\[
\mu
\leftarrow
\max\left\{
{\mu}/{\gamma},
\mu_{\min}
\right\}.
\]
Let $\mu_{\min}>0$ be the uniform lower bound on the accepted step sizes
defined in \eqref{e.mu-min}. Then the number of line-search residual
evaluations at each iteration is bounded above by
\begin{equation}\label{eq:ls_bound_corrected}
C_{\rm LS}
:=
1+
\left\lceil
{
\log(\mu_{\init}/\mu_{\min})
}/{
\log\gamma
}
\right\rceil.
\end{equation}

If at most one additional residual evaluation is required outside the
line search at each iteration, then the total number of residual
evaluations through iteration $\ell$ satisfies $N_{\rm eval}(\ell)
\le
(C_{\rm LS}+1)(\ell+1)$. Consequently,
\[
N_{\rm eval}
=
\begin{cases}
O\!\left(
C_{\rm LS}\varepsilon^{-2}
\right),
&
\text{for the baseline best-iterate guarantee},
\\[1ex]
O\!\left(
C_{\rm LS}\log\varepsilon^{-1}
\right),
&
\text{in the local linear regime}.
\end{cases}
\]
\end{enumerate}
\end{thm}

\begin{proof}
\medskip
\noindent
\textbf{(i) Baseline best-iterate complexity.}
By Theorem~\ref{thm:sublinear_rate}(i), there exists a constant
$C_{\rm sub}>0$ such that, for every $\ell\ge0$, \gzit{e.mininq}
holds, so that $\min_{0\le k\le\ell}
\|E(x_k)\|
\le
\sqrt{{C_{\rm sub}}/{(\ell+1)}
}$. A sufficient condition for the existence of an index
$k_\varepsilon\in\{0,\ldots,\ell\}$ satisfying $\|E(x_{k_\varepsilon})\|
\le
\varepsilon$ is $\sqrt{{C_{\rm sub}}/{(\ell+1)}
}
\le
\varepsilon$, or equivalently, $\ell+1
\ge{C_{\rm sub}}/{\varepsilon^2}$. Thus, choosing $\ell
=
\left\lceil{C_{\rm sub}}/{\varepsilon^2}
\right\rceil
-1$
guarantees the existence of an index $k_\varepsilon\le\ell$ satisfying
the prescribed accuracy. Consequently, $k_\varepsilon
\le
\left\lceil
C_{\rm sub}\varepsilon^{-2}
\right\rceil
-1$, which proves \eqref{eq:iteration_bound_sub}.

This is a best-iterate complexity result. Since the residual sequence
$\{\|E(x_\ell)\|\}_{\ell\ge0}$ need not be nonincreasing, this estimate
does not imply the same $O(\ell^{-1/2})$ bound for every current iterate.

\medskip
\noindent
\textbf{(ii) Local linear-regime complexity.}
Suppose that Assumption {\rm (\Areg)} holds at the limit point
$x^*\in X^*$. From \gzit{e.linearsys}, there exist an
index $\ell_0\ge0$, a constant $q\in(0,1)$, and a constant $M>0$ such
that
$\|E(x_\ell)\|
\le
C_{\rm lin} q^{\,\ell-\ell_0}$  for all $\ell\ge\ell_0$, where $C_{\rm lin}
:=
LM$. A sufficient condition for $\|E(x_\ell)\|
\le
\varepsilon$ is therefore $C_{\rm lin}q^{\,\ell-\ell_0}
\le
\varepsilon$. If $C_{\rm lin}\le\varepsilon$, then $\ell_0$ already satisfies the
prescribed accuracy. Otherwise, since $q\in(0,1)$, the preceding
inequality is equivalent to $\ell-\ell_0
\ge
\D\frac{
\log(C_{\rm lin}/\varepsilon)
}{
\log(1/q)
}$. Consequently, the condition \eqref{eq:iteration_bound_lin}, that is $\ell_\varepsilon
=
O\!\left(
\log\varepsilon^{-1}
\right)$, is obtained from
\[
\ell_\varepsilon
\le
\ell_0+
\max
\left\{
0,
\left\lceil
\frac{
\log(C_{\rm lin}/\varepsilon)
}{
\log(1/q)
}
\right\rceil
\right\}.
\]

\medskip
\noindent
\textbf{(iii) Residual-evaluation complexity.}
At iteration $\ell$, let $j_\ell\ge0$ denote the number of backtracking
reductions performed before the line-search terminates. Before activation
of the safeguard, the $j$th trial step size is
${\mu_{\init}}/{\gamma^j}$. The safeguard update $\mu
\leftarrow
\max\left\{
{\mu}/{\gamma},
\mu_{\min}
\right\}$ ensures that no tested step size is smaller than $\mu_{\min}
$. Therefore, the number of backtracking reductions satisfies
\[
j_\ell
\le
\left\lceil
\frac{
\log(\mu_{\init}/\mu_{\min})
}{
\log\gamma
}
\right\rceil.
\]
Since one residual evaluation is required for each tested trial point,
the number of line-search residual evaluations is $j_\ell+1$ and is
therefore bounded above by $C_{\rm LS}$ in
\eqref{eq:ls_bound_corrected}.

If at most one additional residual evaluation is required outside the
line search at each iteration, then the total number of residual
evaluations through iteration $\ell$ satisfies
\[
N_{\rm eval}(\ell)
\le
(C_{\rm LS}+1)(\ell+1).
\]
Combining this estimate with the iteration bounds in parts~{\rm (i)} and
{\rm (ii)} yields
\[
N_{\rm eval}
=
O\!\left(
C_{\rm LS}\varepsilon^{-2}
\right)
\]
for the baseline best-iterate guarantee and
\[
N_{\rm eval}
=
O\!\left(
C_{\rm LS}\log\varepsilon^{-1}
\right)
\]
in the local linear regime. \qed
\end{proof}

\begin{rem}
The estimates above provide upper bounds on indices at which the
prescribed accuracy is guaranteed. They are worst-case complexity
estimates and do not assert that the residual remains above
$\varepsilon$ at every earlier index. In the baseline regime, the result is a best-iterate statement: there
exists an index $k_\varepsilon$ satisfying
$k_\varepsilon\le
\left\lceil
C_{\rm sub}\varepsilon^{-2}
\right\rceil
-1$
and $
\|E(x_{k_\varepsilon})\|
\le
\varepsilon$. In the local linear regime, the pointwise geometric residual estimate
provides an upper bound on the first index $\ell_\varepsilon$ satisfying $\|E(x_{\ell_\varepsilon})\|
\le
\varepsilon$. Accordingly, the conclusions are upper bounds of the form $k_\varepsilon
\le
\bar\ell_\varepsilon$ or $\ell_\varepsilon
\le
\bar\ell_\varepsilon$, rather than equalities with the corresponding complexity thresholds. The improvement from the
$O(\varepsilon^{-2})$ best-iterate complexity to the
$O(\log(\varepsilon^{-1}))$ pointwise complexity follows from the local
error-bound contraction induced by the exact projection step.
\end{rem}
\section{Numerical results}\label{s.num}

In this section, we investigate the numerical performance of the two
proposed projection-based subspace frameworks: the line-search framework
\texttt{PLS-S} and its fixed-step counterpart \texttt{PF-S}. Both
frameworks employ the Jacobian-free subspace directions introduced in this
paper and differ in their step-size strategy. The \texttt{PLS-S} variants
determine the trial step by the projected line-search procedure, whereas the
\texttt{PF-S} variants use a prescribed fixed step size and therefore do not
invoke a line search.

The numerical experiments have three main purposes. First, we compare the
different two- and three-term subspace directions within both the
\texttt{PLS-S} and \texttt{PF-S} frameworks to identify their most
robust and efficient variants. Second, we compare representative
line-search and fixed-step variants to assess the effect of the step-size
strategy. Third, we compare the strongest proposed variants with
\texttt{EG} and \texttt{OGDA} on the complete test set and examine their
behavior on representative learning-based problems.

Additional background, structural analysis, implementation details, and
numerical results are provided in the supplementary material \cite{suppMat}.
To make the supplementary presentation self-contained, Section~1 of
\cite{suppMat} recalls the nonlinear monotone equation problem considered in
this paper, its connection with convex--concave min--max formulations, and the
MARL, RAL, and GAN application classes. Section~2
examines the asymptotic and canonical \texttt{OGDA}-type direction structures,
including representative \texttt{OGDA}-type methods, the residual--memory
representation of the proposed subspace directions, and the corresponding
unification and convergence-rate implications. Section~3 gives the complete
construction of the $100$ learning-based test problems, including their
dimensions, parameter ranges, known solutions, Lipschitz
constants, initialization procedure, stopping criterion, and numerical
validation. Sections~4--6 report the detailed comparisons of the
\texttt{JFS2dir}, \texttt{JFS3dir1}--\texttt{JFS3dir2}, and
\texttt{JFS3dir3} families, respectively, for both the line-search
\texttt{PLS-S} and fixed-step \texttt{PF-S} realizations. Section~7
reports the efficiency improvements of the fixed-step \texttt{PF-S}
variants over their line-search \texttt{PLS-S} counterparts, while
Section~8 compares the \texttt{OGDA}-type methods, including
\texttt{EG}, \texttt{OGDA}, \texttt{EAG}, and \texttt{FOGDA}. Thus,
\cite{suppMat} provides the structural details and numerical evidence
supporting the theoretical and numerical summaries presented in this paper.
For clarity, throughout the numerical section we use the naming convention
summarized in Tables~\ref{tab:jfs2}--\ref{tab:jfs3dir3} of
Appendix~\ref{app:table}. 
The suffix identifies the particular Jacobian-free subspace direction and
coefficient choice employed by the method.

Specifically,
\texttt{PLS-S-FR}, \texttt{PLS-S-PR}, \texttt{PLS-S-HS},
\texttt{PLS-S-DY}, \texttt{PLS-S-LS}, \texttt{PLS-S-DL}, and
\texttt{PLS-S-HZ} denote the line-search variants associated with the
\texttt{JFS2dir} family, while
\texttt{PF-S-FR}, \texttt{PF-S-PR}, \texttt{PF-S-HS},
\texttt{PF-S-DY}, \texttt{PF-S-LS}, \texttt{PF-S-DL}, and
\texttt{PF-S-HZ} denote their fixed-step counterparts. Similarly,
\texttt{PLS-S-Z1} and \texttt{PLS-S-Z2}, together with
\texttt{PF-S-Z1} and \texttt{PF-S-Z2}, correspond to the
\texttt{JFS3dir1} family; the \texttt{An1}, \texttt{An2}, and
\texttt{De} variants of both frameworks correspond to
\texttt{JFS3dir2}; and
\texttt{PLS-S-A1}--\texttt{PLS-S-A8} together with
\texttt{PF-S-A1}--\texttt{PF-S-A8} correspond to the
\texttt{JFS3dir3} family. Thus, for example,
\texttt{PLS-S-FR} denotes the line-search method employing the
Fletcher--Reeves choice in \texttt{JFS2dir}, whereas
\texttt{PF-S-FR} denotes the corresponding fixed-step method using the
same subspace direction. In the notation \texttt{PLS-S-LS} and
\texttt{PF-S-LS}, the final suffix \texttt{LS} refers to the
Liu--Storey direction and should not be confused with the line-search
mechanism represented by the prefix \texttt{PLS-S}.

For compact presentation in the numerical tables and figures, the
line-search variants are abbreviated by their direction names, for example,
\texttt{LS}, \texttt{PR}, \texttt{Z1}, and \texttt{A7}. The
corresponding fixed-step variants are distinguished by the prefix
\texttt{F-}; for example, \texttt{F-LS}, \texttt{F-PR},
\texttt{F-Z1}, and \texttt{F-A7} denote
\texttt{PF-S-LS}, \texttt{PF-S-PR}, \texttt{PF-S-Z1}, and
\texttt{PF-S-A7}, respectively. This convention is used consistently in
the tables and figures below.

For \texttt{JFS3dir3}, an additional choice concerns the auxiliary vector
$h_\ell$. In \cite{al1}, the choices $h_\ell=E(x_\ell)$ and
$h_\ell=y_{\ell-1}$ are considered. In the present experiments, motivated
by the use of information from the preceding iterate in \texttt{OGDA}, we
instead choose $h_\ell:=E(x_{\ell-1})$. This choice is kept fixed for all
\texttt{JFS3dir3} variants in both the \texttt{PLS-S} and
\texttt{PF-S} frameworks.

For numerical stability, the curvature denominators appearing in the subspace coefficients are safeguarded whenever they are zero or sufficiently close to zero in finite-precision arithmetic. In particular, the safeguards are applied to the denominators involving $y_{\ell-1}^Ts_{\ell-1}$ and $y_{\ell-1}^Tp_{\ell-1}$ before the corresponding coefficients are evaluated. The precise safeguarding construction and its well-definedness are given in Proposition~\ref{pr3} of Appendix~\ref{app:subparam}. After safeguarding, the usual coefficient bounds, direction scaling, and angle condition are imposed as in Algorithm~\ref{a.FDS}.

The remainder of this section is organized as follows. We first specify the
tuning parameters used by the proposed and comparison methods and introduce
the tools used to assess robustness and computational efficiency. The
complete construction and validation of the learning-based test problems,
as well as the detailed comparisons within the individual
\texttt{JFS2dir}, \texttt{JFS3dir1}, \texttt{JFS3dir2}, and
\texttt{JFS3dir3} families and among the \texttt{OGDA}-type methods, are
reported in the supplementary material \cite{suppMat}. Here, we focus on
the most robust or efficient representatives identified by those detailed
comparisons. We compare selected \texttt{PLS-S} and \texttt{PF-S}
variants, assess the effect of the step-size strategy, and then compare the
most competitive proposed method with \texttt{EG} and \texttt{OGDA}.
Finally, representative learning-based problems are used to illustrate the
convergence behavior of the methods with respect to both the NME merit
function and the objective-value error of the underlying min--max problems.

The supplementary material, MATLAB implementations of the proposed
\texttt{PLS-S} and \texttt{PF-S} methods, test-equation collection, and
numerical examples used in the experiments are available in
\cite{suppMat}.

\subsection{Tuning parameters}
\label{subsec:tuneParam}

We next specify the tuning parameters used in the numerical experiments.
To make the experiments reproducible, we distinguish the parameters of
the proposed \texttt{PLS-S} and \texttt{PF-S} methods from those of the comparison methods
\texttt{EG} and \texttt{OGDA}. Throughout this section, we use the
notation introduced in Section~\ref{sec.new}; in particular, the
implementation-specific variable names are replaced by their mathematical
counterparts.

\bfi{Parameters of \texttt{PLS-S}.}
The projected line search in Step~{\rm(S1$_a$)} of
Algorithm~\ref{a.FDS} is determined by the sufficient-descent parameter
\(\rho\), the backtracking factor \(\gamma\), and the initial trial
step size \(\mu_{\init}\). In all experiments, we set
\[
    \rho=10^{-2},
    \qquad
    \gamma=2,
    \qquad
    \mu_{\init}={0.48}/{L},
\]
where \(L\) is the Lipschitz constant associated with the corresponding
test problem. Thus, whenever the initial trial step size is rejected, the
line search successively reduces it by the factor
\(1/\gamma=1/2\).

The residual-based angle condition \eqref{e.rac} is imposed with
\(\Delta_{\angle}=0.9\). Whenever the candidate direction
fails this condition, it is modified according to
\eqref{e.sigma}--\eqref{e.pcg2}, and the resulting direction is subsequently scaled as prescribed
in Step~{\rm(S5)} of Algorithm~\ref{a.FDS}.

The lower safeguard for the line-search step size is then chosen according to
\eqref{e.mu-min}, namely,
\[
    \mu_{\min}
    =
    \min\left\{
        \mu_{\init},
        \frac{\Delta_{\angle}}
             {\gamma(L+\rho)}
    \right\}
    =
    \min\left\{
        \frac{0.48}{L},
        \frac{0.9}
             {2(L+10^{-2})}
    \right\}.
\]
Consequently, \(\mu_{\min}\) is not treated as an independently tuned
parameter; it is determined by \(L\) together with
\(\Delta_{\angle}\), \(\gamma\), and \(\rho\), in accordance with the
line-search analysis in Section~\ref{sec.new}.

\bfi{Parameters of \texttt{PF-S}.}
For the fixed-step realization \texttt{PF-S}, the line-search procedure is
not invoked. Instead, Step~{\rm(S1$_b$)} of Algorithm~\ref{a.FDS} uses a
constant step size \(\mu_{\ell}=\alpha\) satisfying
\eqref{e.alpha_bound}. In the numerical experiments, we use
\[
    \alpha
    =
    \min\left\{
        1,
        \frac{0.9}
             {L+10^{-2}}
    \right\},
\]
where \(L\) is the Lipschitz constant associated with the corresponding
test problem. Thus, \(\alpha\) is selected as the largest step size allowed
by the admissibility condition \eqref{e.alpha_bound}. The same
sufficient-descent and angle parameters as for \texttt{PLS-S} are used,
namely, $\rho=10^{-2}$ and $\Delta_{\angle}=0.9$. Hence, the sufficient-descent inequality
\eqref{e.bls} holds directly and no backtracking procedure is required.

For the controlled scaling of the subspace directions in
\eqref{e.scaleDir}, we use \(\zeta=10^{-3}\). Hence, after scaling, every
nonzero search direction satisfies \eqref{e:sccon}. This scaling rule and
the residual-based angle safeguard are used for both \texttt{PLS-S} and
\texttt{PF-S}.

The particular member of the \texttt{PLS-S} or \texttt{PF-S} class is determined by the
choice of the Jacobian-free subspace direction in
\eqref{e.pcg1}. Thus, \texttt{FR}, \texttt{PR}, \texttt{HS},
\texttt{DY}, \texttt{LS}, \texttt{HZ}, and \texttt{DL} select members
of the two-term \texttt{JFS2dir} family; \texttt{Z1} and \texttt{Z2}
select the two members of \texttt{JFS3dir1}; \texttt{An1},
\texttt{An2}, and \texttt{De} select the members of
\texttt{JFS3dir2}; and \texttt{A1}--\texttt{A8} select the members of
\texttt{JFS3dir3}. These choices specify the subspace formula rather
than a continuous tuning parameter.

For the \texttt{JFS3dir3} family, the auxiliary parameters entering the
definitions of the coefficients are fixed at
\(\bar\gamma=0.8\), \(\varsigma_1=10^{-2}\), and
\(\varsigma_2=10^{2}\). Accordingly, the corresponding coefficient
construction uses
\[
    \gamma_\ell
    =
    \max\left\{
        \varsigma_1,
        \min\{\varsigma_2,\bar\gamma_\ell\}
    \right\}.
\]
For the \texttt{DL} direction, the constant appearing in the modified
residual difference is fixed at \(0.1\). These constants are kept
unchanged in all experiments in which the corresponding direction is
used. 

For the safeguarded subspace coefficients of all directions, we use $\beta_{\min}=-10^{50}$ and $
\beta_{\max}=10^{50}$.

\bfi{Parameters of \texttt{EG} and \texttt{OGDA}.}
For the comparison methods, we use the unconstrained versions of
\texttt{EG} and \texttt{OGDA}.
Their principal tuning parameter is the step size \(s>0\).
For \texttt{EG}, the iterations are
\[
    w_k=x_k-sE(x_k),
    \qquad
    x_{k+1}=x_k-sE(w_k),
\]
whereas \texttt{OGDA} uses
\[
    x_{k+1}
    =
    x_k-2sE(x_k)+sE(x_{k-1}),
\]
after the initial step \(x_1=x_0-sE(x_0)\).
In the numerical experiments, the step size supplied to both methods is $s=0.48/L$, where \(L\) is the Lipschitz constant associated with the corresponding
test problem. Consistently with the admissible step-size regimes of the two
methods, the implementation additionally uses the safety caps
\[
    s\le{0.99}/{L}
    \qquad\text{for \texttt{EG}},
    \qquad
    s\le{0.49}/{L}
    \qquad\text{for \texttt{OGDA}}.
\]
Since \(0.48/L\) lies below both caps, neither safeguard is activated in the
reported experiments. Consequently, the effective step size used by both
methods is
\[
    s_{\rm EG}=s_{\rm OGDA}={0.48}/{L}.
\]
Thus, \texttt{EG} and \texttt{OGDA} are compared using the same Lipschitz-scaled step size, while the implementation retains method-specific
safety caps consistent with the strict step-size requirements.

\subsection{Tools for Efficiency and Robustness}

We denote by $\mathcal S$ the set of solvers and by $\mathcal P$ the set of test problems. A solver $s\in\mathcal S$ is regarded as efficient on a collection of test problems if it achieves a low computational cost relative to the other solvers. To assess the relative efficiency of the solvers, we use the performance profiles of Dolan and Mor{\'e} \cite{DolM}, together with a scalar efficiency measure introduced below.

Let $c_{s,p}$ denote the {\bf cost measure} associated with solver
$s\in\mathcal S$ on problem $p\in\mathcal P$. For a problem that is
successfully solved, the {\bf performance ratio} is defined by
\[
pfr_{s,p}
:=
\frac{c_{s,p}}
{\displaystyle\min_{\ol{s}\in\mathcal S}c_{\ol{s},p}}.
\]
The corresponding {\bf performance profile} of solver $s$ is
\[
\rho_s(\tau)
:=
\frac{1}{|\mathcal P|}
\left|
\left\{
p\in\mathcal P
\,\middle|\,
pfr_{s,p}\le\tau
\right\}
\right|,
\qquad \tau\ge 1.
\]
Thus, $\rho_s(\tau)$ gives the fraction of test problems for which
solver $s$ performs within a factor $\tau$ of the best solver with
respect to the given cost measure. In particular, $\rho_s(1)$ represents the fraction of problems for which solver $s$ attains the best value of the corresponding cost measure. As usual, unsuccessful runs are assigned an infinite performance ratio and hence do not contribute to $\rho_s(\tau)$ for any finite $\tau$.

In addition to performance profiles, we use a scalar efficiency measure \cite{VRBBO}
to summarize the relative performance of each solver over the entire
test set. For a given solver $s\in\mathcal S$ and problem
$p\in\mathcal P$, its {\bf efficiency} with respect to the given cost
measure is defined by
\[
e_{s,p}
:=
\left\{
\begin{array}{ll}
\displaystyle
100\,
\frac{\min_{\ol{s}\in\mathcal S}c_{\ol{s},p}}
     {c_{s,p}},
& \hbox{if solver $s$ solves problem $p$},\\[3mm]
0,
& \hbox{otherwise.}
\end{array}
\right.
\]
Hence, for every successfully solved problem, $e_{s,p}
={100}/{pfr_{s,p}}$, and therefore $0<e_{s,p}\le 100$. A value of $100$ means that solver $s$ attains the lowest cost among all
solvers on problem $p$, whereas smaller positive values indicate a higher cost relative to the best solver. An unsuccessful run is assigned zero efficiency.

Although the efficiency measure and the performance profile are both
based on the same relative cost, they summarize the results differently.
The performance profile $\rho_s(\tau)$ describes the distribution of
the performance ratios over the test set and, consequently, shows the
fraction of problems for which a solver lies within a prescribed factor
of the best solver. In contrast, $e_{s,p}$ converts the performance
ratio for each problem into a percentage score.

For each solver $s\in\mathcal S$, we define its {\bf mean efficiency}
over the complete test set by
\[
e_{T,s}
:=
({1}/{|\mathcal P|})
\sum_{p\in\mathcal P}e_{s,p}; \ \ \text{equivalently}, \ \ e_{T,s}
=
({100}/{|\mathcal P|})
\sum_{\substack{p\in\mathcal P\\
                 s\ {\rm solves}\ p}}
({1}/{pfr_{s,p}}),
\]
since unsolved problems contribute zero to the sum. Consequently,
\[
0\le e_{T,s}\le 100.
\]
A larger value of $e_{T,s}$ indicates better overall efficiency over
the complete test set. Notice that $e_{T,s}$ accounts simultaneously
for efficiency and successful solution of the problems: failure to solve
a problem contributes zero to the mean efficiency.

For the tabulated comparisons, we consider the number of function
evaluations, denoted by \texttt{nf}, and the computational time in
milliseconds, denoted by \texttt{msec}, as the two cost measures.
Accordingly, the corresponding mean efficiencies are denoted by $e_{T,s}^{\tt nf}$ and $e_{T,s}^{\tt msec}$, respectively. All efficiencies reported in the tables are expressed as
percentages and are rounded to integers.

We also use the number of successfully solved problems as a measure of
{\bf robustness}. More precisely, problem $p$ is said to be {\bf solved}
by solver $s$ if $f\le 10^{-5}$, and is said to be {\bf unsolved} otherwise. Thus, a solver that
successfully solves a larger number of problems is regarded as more
robust on the test set.

For additional information about unsuccessful runs, the tables report
three anomaly counters. For a given solver $s$, $\#n$ denotes the number
of problems for which the prescribed maximum number of evaluations is
reached, $\#t$ denotes the number of problems for which the prescribed
maximum computational time is reached, and $\#f$ denotes the number of
problems for which the solver fails, encounters an error, or terminates
because of another algorithmic termination condition. In the numerical
experiments, the stopping limits are ${\tt nf}\le 20000$ and ${\tt sec}\le 300$.  Together, the number of solved problems, the anomaly counters, and the
mean efficiencies $e_{T,s}^{\tt nf}$ and $e_{T,s}^{\tt msec}$ provide
complementary measures of robustness and computational efficiency.

\subsection{A comparison between most robust or efficient}\label{sec:allres}

Table~\ref{tab:allres} compares selected variants that are either the most
robust or the most efficient within their respective groups, as identified in
the detailed comparisons reported in \cite[Tables~3--6]{suppMat}, while
Figure~\ref{f.f} reports the corresponding performance profiles for
\texttt{PLS-S-Z2}, \texttt{OGDA}, and \texttt{EG} in terms of
\texttt{nf} and \texttt{msec}.

\begin{center}
\begin{adjustbox}{center, width=\columnwidth-1pt}
\footnotesize
\begin{tabular}{? l? l? r?? r? r? r?? r? r?}
\hline
\multicolumn{1}{?l}{stopping test:~~~~~~} &  \multicolumn{7}{c?}{$f \le$ {\tt 1e-05,}~~~~~~~~~~~ ${\tt sec}\le$ {\tt 300,}~~~~~~~~~~~ ${\tt nf}\le$ {\tt 20000}}\\
\hline
\multicolumn{3}{?l??}{100 of 100 problems solved} &\multicolumn{3}{r??}{} & \multicolumn{2}{c?}{}\\
\cline{1-8}
\multicolumn{3}{?l??}{dim$\in$[8,135]}&\multicolumn{3}{r??}{\# of anomalies} &\multicolumn{2}{c?}{$e_{T,s}$ in \%}\\
\hline
\multicolumn{2}{?l?}{solver} & solved&\#n&\#t&\#f&$e_{T,s}^{\tt nf}$&$e_{T,s}^{\tt msec}$\\
\hline
{\bf PLS-S-Z2}& {\bf Z2}&\cellcolor[HTML]{C0C0C0}100&0&0&0&\cellcolor[HTML]{C0C0C0}94&\cellcolor[HTML]{C0C0C0}94\\
{\bf PF-S-LS}& {\bf F-LS}&\cellcolor[HTML]{C0C0C0}100&0&0&0&44&46\\
{\bf PF-S-A4}& {\bf F-A4}&\cellcolor[HTML]{C0C0C0}100&0&0&0&40&50\\
\hline
\multicolumn{3}{?l??}{100 of 100 problems solved} &\multicolumn{3}{r??}{} & \multicolumn{2}{c?}{}\\
\cline{1-8}
\multicolumn{3}{?l??}{dim$\in$[8,135]}&\multicolumn{3}{r??}{\# of anomalies} &\multicolumn{2}{c?}{$e_{T,s}$ in \%}\\
\hline
\multicolumn{2}{?l?}{solver} & solved&\#n&\#t&\#f&$e_{T,s}^{\tt nf}$&$e_{T,s}^{\tt msec}$\\
\hline
{\bf PLS-S-Z2}& {\bf Z2}&\cellcolor[HTML]{C0C0C0}100&0&0&0&\cellcolor[HTML]{C0C0C0}80&\cellcolor[HTML]{C0C0C0}95\\
{\bf OGDA}& {\bf OGDA}&\cellcolor[HTML]{C0C0C0}100&0&0&0&61&41\\
{\bf EG}& {\bf EG}&\cellcolor[HTML]{C0C0C0}100&0&0&0&30&14\\
\hline
\end{tabular}
\end{adjustbox}

\captionof{table}{Comparison between the most robust or efficient variants.}
\label{tab:allres}

\medskip
\end{center}

The first block of Table~\ref{tab:allres} compares three selected variants
of the proposed projection-based subspace framework. All three methods,
\texttt{PLS-S-Z2}, \texttt{PF-S-LS}, and \texttt{PF-S-A4}, solve the
entire test set without anomalies. Among them, \texttt{PLS-S-Z2} is
clearly the most efficient in both measures, attaining
$e_{T,s}^{\tt nf}=94\%$ and $e_{T,s}^{\tt msec}=94\%$.
The fixed-step variants \texttt{PF-S-LS} and \texttt{PF-S-A4} attain
residual-evaluation efficiencies of $44\%$ and $40\%$, respectively,
while their computational-time efficiencies are $46\%$ and $50\%$,
respectively. Thus, within this direct comparison,
\texttt{PLS-S-Z2} provides the strongest overall performance.

The second block of Table~\ref{tab:allres} directly compares
\texttt{PLS-S-Z2} with the benchmark methods \texttt{OGDA} and
\texttt{EG}. All three methods solve all $100$ problems without anomalies.
However, \texttt{PLS-S-Z2} is substantially more efficient, attaining
$e_{T,s}^{\tt nf}=80\%$ and $e_{T,s}^{\tt msec}=95\%$. The corresponding
efficiencies of \texttt{OGDA} are $61\%$ and $41\%$, whereas
\texttt{EG} attains $30\%$ and $14\%$, respectively. In particular,
\texttt{OGDA} is more competitive than \texttt{EG} in terms of residual
evaluations, but neither benchmark method matches the overall efficiency
of \texttt{PLS-S-Z2}.

\begin{figure}[http!]
	\centering
	\begin{tabular}{cc}
		\includegraphics[width=0.47\textwidth]{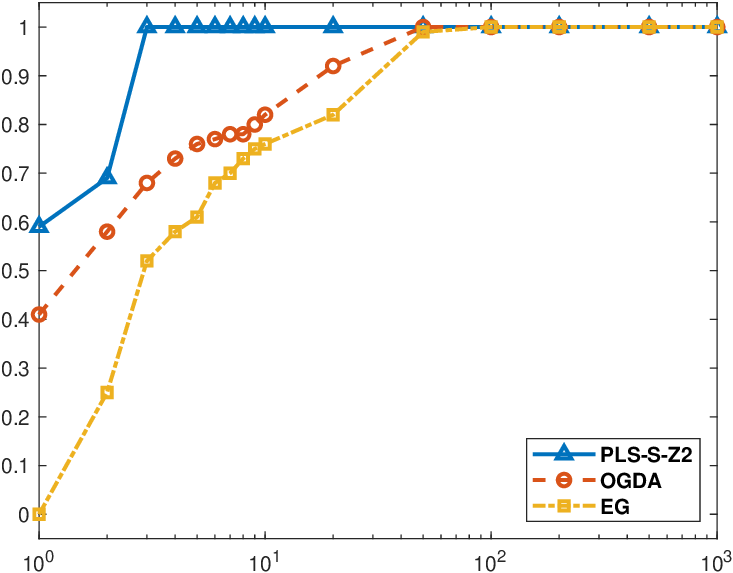}
		&
		\includegraphics[width=0.47\textwidth]{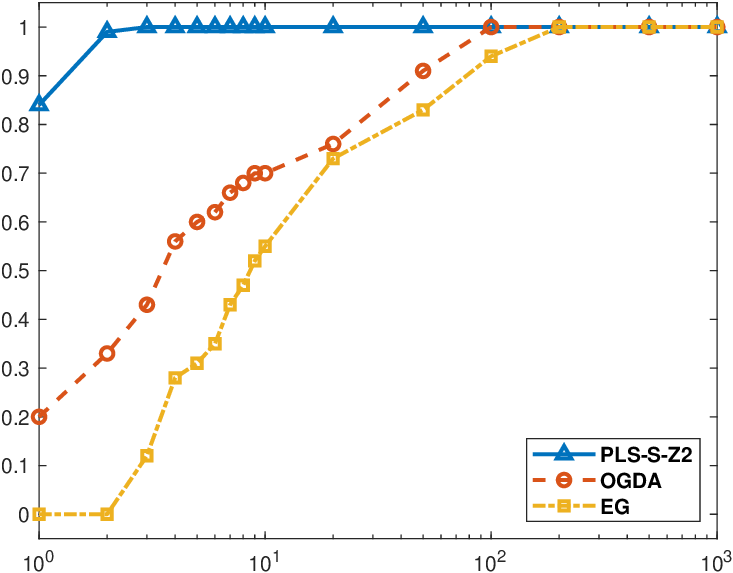}
	\end{tabular}
	\caption{Performance profiles of \texttt{PLS-S-Z2}, \texttt{OGDA}, and \texttt{EG} in terms of \texttt{nf} (left) and \texttt{msec} (right). The performance profile plots \(\rho(\tau)\) versus the performance ratio \(\tau\).}
	\label{f.f}
\end{figure}

\begin{figure}[http!]
	\centering
	\begin{tabular}{cc}
		\includegraphics[width=0.47\textwidth]
		{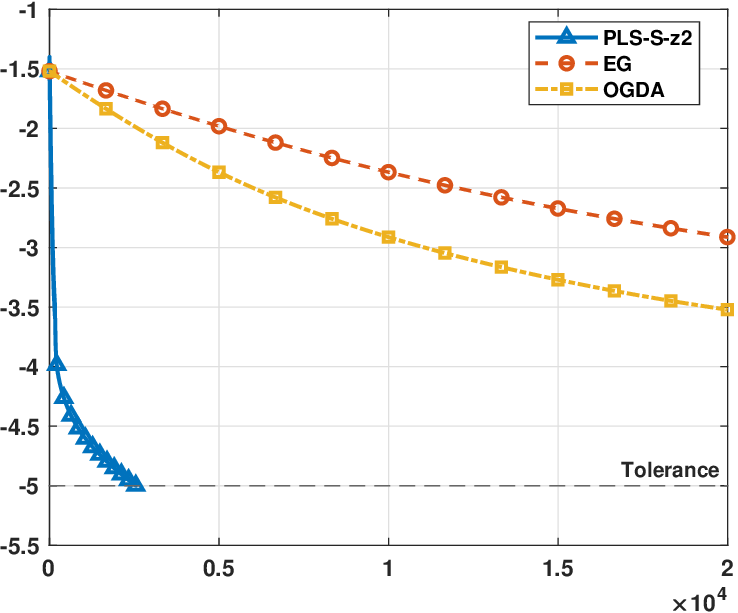}
		&
		\includegraphics[width=0.47\textwidth]
		{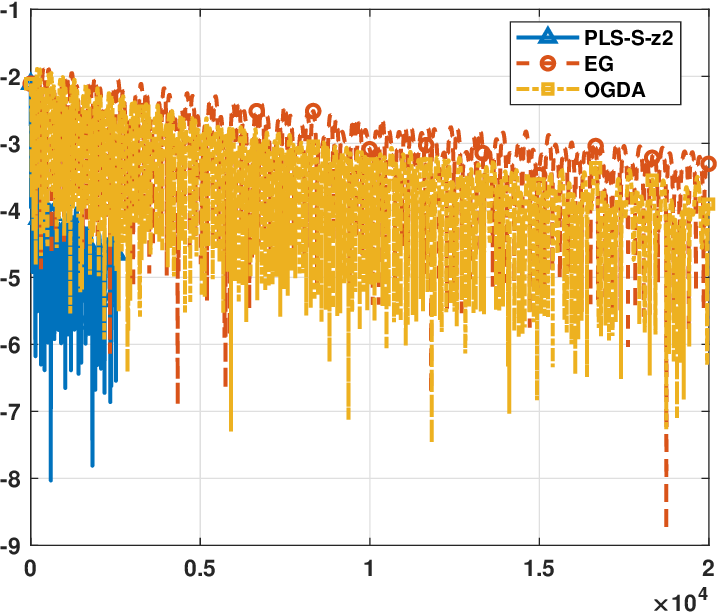}
		\\[1mm]
		\includegraphics[width=0.47\textwidth]
		{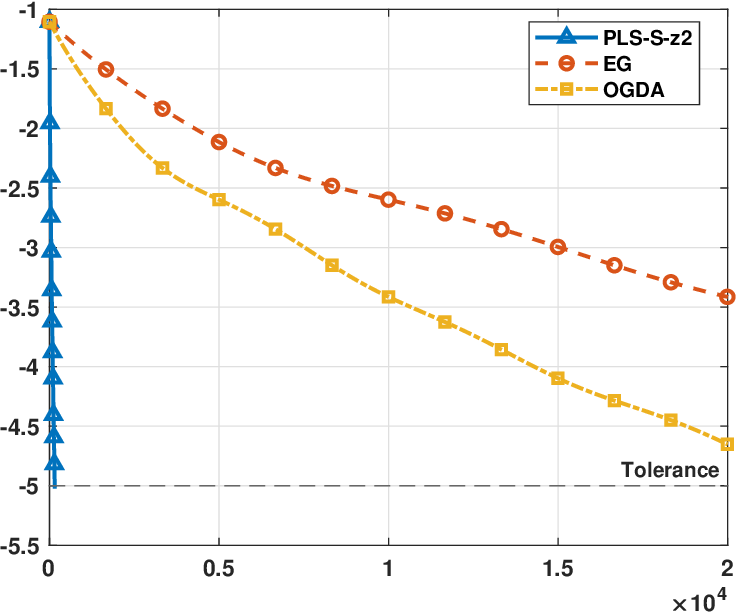}
		&
		\includegraphics[width=0.47\textwidth]
		{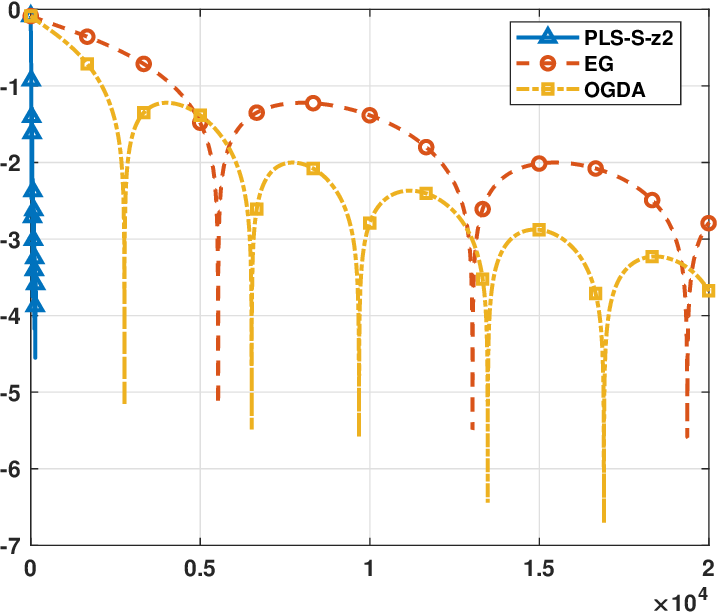}
		\\[1mm]
		\includegraphics[width=0.47\textwidth]
		{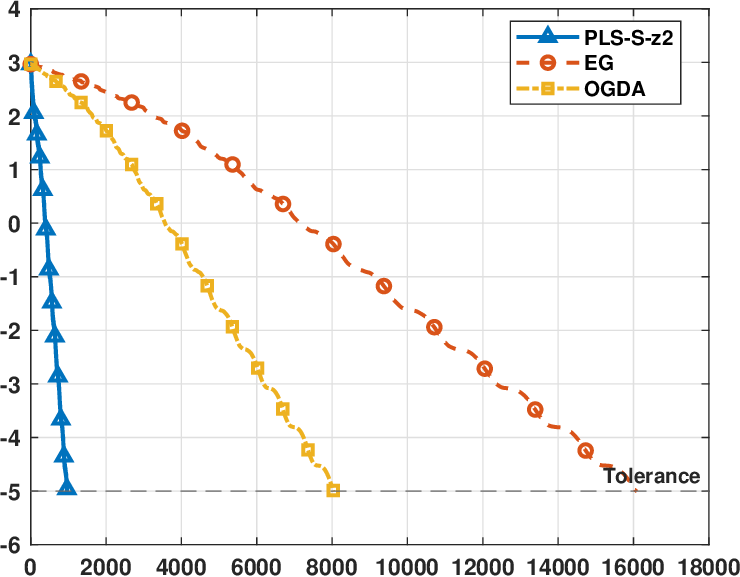}
		&
		\includegraphics[width=0.47\textwidth]
		{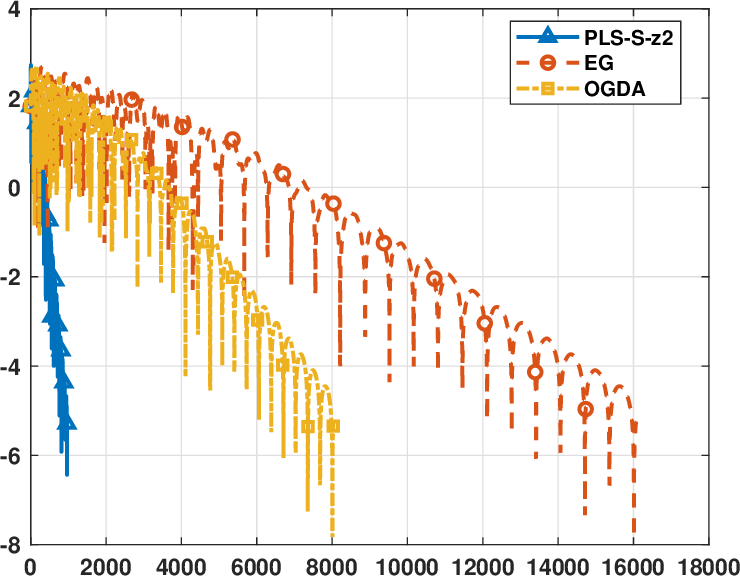}
	\end{tabular}
	\caption{Convergence histories of \texttt{PLS-S-Z2}, \texttt{OGDA}, and \texttt{EG} for representative learning-based test problems from MARL (top row),
	RAL (middle row), and GANs (bottom row).
	The left column reports
	\(\log_{10}\!\left(f\right)\), corresponding to the
	merit function used in the stopping criterion, while the right column
	reports the min--max objective-value error
	\(\log_{10}|\Phi(x)-\Phi(x^*)|\), where \(\Phi\) denotes the objective
function of the underlying min--max problem and \(x^*\) is its known solution.
In all panels, the horizontal axis is the number \texttt{nf} of function evaluations. The horizontal reference level in the left-column panels
	corresponds to the stopping criterion
	\(f\le 10^{-5}\).}
	\label{f.learning}
\end{figure}

Figure~\ref{f.f} complements the mean-efficiency results through
performance profiles. In terms of both \texttt{nf} and \texttt{msec},
the profile of \texttt{PLS-S-Z2} rises more rapidly than those of
\texttt{OGDA} and \texttt{EG} and remains above them over most of the
displayed range. In particular, \texttt{PLS-S-Z2} has the largest
fraction of wins at small performance ratios and reaches high values of
$\rho(\tau)$ substantially earlier. The \texttt{OGDA} profile is
generally stronger than that of \texttt{EG}, especially in terms of
residual evaluations, consistently with the mean efficiencies reported
in Table~\ref{tab:allres}. Since all three methods solve the complete
test set, their profiles eventually reach one, but
\texttt{PLS-S-Z2} exhibits the clearest overall efficiency advantage.

Figure~\ref{f.learning} compares the convergence behavior of
\texttt{PLS-S-Z2}, \texttt{OGDA}, and \texttt{EG} on representative
problems from the three learning-based applications. The left column
reports the NME merit function
$f=\frac12\|E(x)\|^2$, whereas the right column reports the error in the
objective value of the corresponding min--max problem. Thus, the two
columns provide complementary views of the numerical behavior: the former
measures convergence with respect to the nonlinear monotone equation
actually solved by the algorithms, while the latter illustrates the
corresponding evolution of the original min--max objective value.

For the MARL problem, \texttt{PLS-S-Z2} reaches the prescribed NME
stopping tolerance after $2553$ residual evaluations. In contrast,
\texttt{OGDA} and \texttt{EG} both exhaust the maximum budget of $20000$
evaluations without satisfying the prescribed tolerance. Their best NME
merit values are approximately $3.01\times10^{-4}$ and
$1.22\times10^{-3}$, respectively, whereas
\texttt{PLS-S-Z2} reaches $9.99\times10^{-6}$. At the corresponding
best-merit points, the objective-value errors are approximately
$2.34\times10^{-5}$ for \texttt{PLS-S-Z2},
$1.35\times10^{-4}$ for \texttt{OGDA}, and
$4.76\times10^{-4}$ for \texttt{EG}. Thus, on this example,
\texttt{PLS-S-Z2} is the only method among the three that reaches the
prescribed NME accuracy within the available evaluation budget.

The advantage of \texttt{PLS-S-Z2} is even more pronounced for the RAL
problem. It reaches the prescribed NME tolerance after only $156$ residual
evaluations. Both \texttt{OGDA} and \texttt{EG} exhaust the maximum
budget of $20000$ evaluations without satisfying the stopping criterion.
The best NME merit values obtained by \texttt{OGDA} and \texttt{EG} are
approximately $2.22\times10^{-5}$ and $3.85\times10^{-4}$, respectively,
whereas \texttt{PLS-S-Z2} reaches approximately
$9.44\times10^{-6}$. The corresponding objective-value errors are
approximately $1.46\times10^{-4}$ for \texttt{PLS-S-Z2},
$2.09\times10^{-4}$ for \texttt{OGDA}, and
$1.63\times10^{-3}$ for \texttt{EG}. These results again illustrate
that the NME merit function and the original objective-value error provide
different measures of progress and need not decrease at the same rate.

For the GAN problem, all three methods reach the prescribed NME tolerance,
but the required numbers of residual evaluations differ substantially.
The method \texttt{PLS-S-Z2} terminates after $965$ evaluations, compared
with $8035$ for \texttt{OGDA} and $16068$ for \texttt{EG}. Their final
merit values are approximately $9.73\times10^{-6}$,
$9.99\times10^{-6}$, and $9.99\times10^{-6}$, respectively. The
corresponding objective-value errors are approximately
$1.91\times10^{-6}$ for \texttt{PLS-S-Z2} and
$5.65\times10^{-6}$ for both \texttt{OGDA} and \texttt{EG}.
Hence, although all three methods eventually satisfy the prescribed
accuracy on this instance, \texttt{PLS-S-Z2} requires substantially fewer
residual evaluations.

Taken together, the three representative learning problems consistently
favor \texttt{PLS-S-Z2} under the NME stopping criterion. It reaches the
prescribed tolerance on all three examples, whereas the benchmark methods
fail to do so within the evaluation budget on the MARL and RAL examples.
On the GAN problem, where all three methods succeed,
\texttt{PLS-S-Z2} still requires substantially fewer residual evaluations.
At the same time, the objective-value histories demonstrate that the NME
merit function and the original min--max objective error measure different
aspects of the iterative process and therefore need not decrease at
comparable rates. Accordingly, the NME merit function provides the
appropriate basis for assessing convergence of the nonlinear monotone
equation, while the objective-value error serves as a complementary
measure of progress in the underlying learning formulation.

Overall, the results in Table~\ref{tab:allres} and
Figures~\ref{f.f}--\ref{f.learning} consistently identify
\texttt{PLS-S-Z2} as the strongest overall representative among the
methods compared. In the first block of Table~\ref{tab:allres}, it
outperforms the selected fixed-step variants \texttt{PF-S-LS} and
\texttt{PF-S-A4} in both efficiency measures, while in the second block
it substantially outperforms the benchmark methods \texttt{OGDA} and
\texttt{EG}. The representative learning-based experiments lead to the
same qualitative conclusion and further illustrate the practical
advantage of the line-search \texttt{PLS-S-Z2} realization of the
proposed projection-based subspace framework.

\section{Conclusion}

This paper developed a projection-based subspace framework for solving
nonlinear monotone equations. The proposed {\tt PLS-S} and {\tt PF-S} methods combine four
families of Jacobian-free subspace directions with either fixed step sizes or
variable step sizes generated by the projected line search of Solodov and
Svaiter.

The convergence analysis is independent of the specific algebraic form of the
raw subspace directions. Instead, it relies on an enforced residual-based
angle condition, an explicit scaling rule, and projection onto separating
half-spaces. Under monotonicity and Lipschitz continuity, these ingredients
yield global convergence without requiring cocoercivity, global strong
monotonicity, merit functions, or boundedness of the unscaled directions.

For the exact projection variant, we established the best-iterate residual
rate $\min_{0\le k\le\ell}\|E(x_k)\|
=
O(\ell^{-1/2})$. Under a local error bound, the distance to the solution set satisfies $\dist(x_\ell,X^*)
=
o(\ell^{-1/2})$. Under the stronger local Jacobian regularity condition at the limit solution,
which induces the required local error bound and local isolation of the limit
solution, the iterates converge locally $R$-linearly. Consequently, the
last-iterate residual converges geometrically and satisfies $\|E(x_\ell)\|
=
o(\ell^{-1})$. We also derived explicit iteration and residual-evaluation complexity bounds.
The baseline best-iterate guarantee requires $O(\varepsilon^{-2})$ iterations, while the local linear regime yields $O\!\left(\log\varepsilon^{-1}\right)$ pointwise complexity.

The numerical results support the practical effectiveness of both the
line-search \texttt{PLS-S} and fixed-step \texttt{PF-S} realizations of the
proposed framework. Several variants from both classes solved the complete
test set without anomalies. Among the selected representatives,
\texttt{PLS-S-Z2} showed the strongest overall performance, attaining the
highest mean efficiencies in both residual evaluations and computational
time. Among the selected fixed-step variants, \texttt{PF-S-LS} was more
efficient in terms of residual evaluations, whereas \texttt{PF-S-A4}
attained the higher time efficiency. The direct comparison with
\texttt{EG} and \texttt{OGDA} further favored \texttt{PLS-S-Z2}, which
achieved substantially larger mean efficiencies in both residual evaluations
and computational time while all three methods solved the complete test set.
The performance profiles likewise showed a clear practical advantage for
\texttt{PLS-S-Z2}. Finally, on the representative MARL, RAL, and GAN
problems, \texttt{PLS-S-Z2} reached the prescribed NME stopping tolerance
on all three instances. On the MARL and RAL problems, both \texttt{EG} and
\texttt{OGDA} exhausted the maximum evaluation budget without reaching the
prescribed tolerance, whereas on the GAN problem \texttt{PLS-S-Z2} required
substantially fewer residual evaluations than both benchmark methods. The
experiments also confirm that the NME merit function and the objective-value
error measure different aspects of the iterative process and need not
decrease at comparable rates.

\appendix

\section{Appendix: Projection inequality (exact projection)}
\label{app:Relaxed}

In this appendix, we recall the classical descent property of the metric
projection onto a nonempty closed convex set in a Hilbert space. 
The result follows directly from the projection characterization 
\cite[Theorem~3.16]{Bauschke2017}, the polarization identity 
\cite[Lemma~2.12]{Bauschke2017}, and the firm nonexpansiveness of projectors 
\cite[Proposition~4.16]{Bauschke2017}. 
For completeness, we provide a self-contained proof.

Let $C \subset H$ be nonempty, closed, and convex. 
For $x \in H$, we denote by $P_C(x)$ the (unique) metric projection of $x$ onto $C$, i.e., $P_C(x) := \arg\min_{y \in C} \|x-y\|$. Existence and uniqueness are guaranteed by 
\cite[Theorem~3.16]{Bauschke2017}.

\begin{prop}[Projection inequality]\label{p.Projection} 
Let $H$ be a real Hilbert space and let $C \subset H$ be nonempty, closed, and convex. 
Let $x \in H$ and define $x^{+} := P_C(x)$. Then for every $z \in C$ it holds $\|x^{+} - z\|^2
\le
\|x - z\|^2
-
\|x - x^{+}\|^2$.
\end{prop}

\begin{proof}
We have $x^{+}=P_C(x)$. 
So, by \cite[Theorem~3.16]{Bauschke2017}, 
\begin{equation}\label{e.projAp}
 x^{+} \in C 
\quad\text{and}\quad
\langle z - x^{+}, x - x^{+}\rangle \le 0
\quad \forall z \in C.   
\end{equation}
For every $z\in C$, we have
\[
\begin{aligned}
\|x-z\|^2
&=
\|x-x^+\|^2
+
\|x^+-z\|^2
+
2\langle x-x^+,x^+-z\rangle
\ge
\|x-x^+\|^2
+
\|x^+-z\|^2,
\end{aligned}
\]
because \eqref{e.projAp} gives $\langle x-x^+,x^+-z\rangle\ge0$. Rearranging yields the claimed inequality. \qed
\end{proof}

\section{Appendix: Subspace parameter choices}\label{app:table}

This appendix summarizes the parameter choices used in the various
{\tt PLS-S} and  {\tt PF-S} implementations. For clarity, the methods are grouped
according to the underlying subspace direction structure. Throughout
this appendix, we use the auxiliary quantities defined in \eqref{e.sy},
namely,
\begin{equation}\label{e.aux}
y_{\ell-1}
=
E(\bar x_{\ell-1})-E(x_{\ell-1}),
\qquad
s_{\ell-1}
=
\bar x_{\ell-1}-x_{\ell-1}=\mu_{\ell-1}p^{\scal}_{\ell-1}.
\end{equation}
The tables below provide a structured overview of all parameter variants
used in the numerical experiments. Organizing the methods according to
their direction structure highlights their algebraic relationships while
preserving clarity and readability.

\subsection{{\tt JFS2}-based methods}

For this group of subspace directions, define $\bar y_{\ell-1}
=
y_{\ell-1}
-
2\D\frac{\|y_{\ell-1}\|^2}{y_{\ell-1}^Tp_{\ell-1}}p_{\ell-1}$ and $\wt y_{\ell-1}
=
y_{\ell-1}
-t p_{\ell-1}$ with $t=0.1$.

\begin{table}[H]
\centering
\caption{JFS2-based methods and corresponding parameters.}
\label{tab:jfs2}
\renewcommand{\arraystretch}{1.3} 
\begin{tabular}{l c}
\hline
\textbf{Method \& Direction} & $\boldsymbol{\beta_\ell}$ \\ \hline
{\tt PLS-S-FR} and {\tt PF-S-FR} / {\tt JFS2dir-FR} & $\frac{\|E(x_\ell)\|^2}{\|E(x_{\ell-1})\|^2}$ \\[1ex]
{\tt PLS-S-PR} and {\tt PF-S-PR} / {\tt JFS2dir-PR} & $\frac{y_{\ell-1}^T E(x_\ell)}{\|E(x_{\ell-1})\|^2}$ \\[1ex]
{\tt PLS-S-HS} and {\tt PF-S-HS} / {\tt JFS2dir-HS} & $\frac{y_{\ell-1}^T E(x_\ell)}{y_{\ell-1}^T p_{\ell-1}}$ \\[1ex]
{\tt PLS-S-DY} and {\tt PF-S-DY} / {\tt JFS2dir-DY} & $\frac{\|E(x_\ell)\|^2}{y_{\ell-1}^T p_{\ell-1}}$ \\[1ex]
{\tt PLS-S-LS} and {\tt PF-S-LS} / {\tt JFS2dir-LS} & $-\frac{E(x_\ell)^T y_{\ell-1}}{E(x_{\ell-1})^T p_{\ell-1}}$ \\[1ex]
{\tt PLS-S-DL} and {\tt PF-S-DL} / {\tt JFS2dir-DL} & $\frac{E(x_\ell)^T \bar y_{\ell-1}}{y_{\ell-1}^T p_{\ell-1}}$ \\[1ex]
{\tt PLS-S-HZ} and {\tt PF-S-HZ} / {\tt JFS2dir-HZ} & $\frac{E(x_\ell)^T\wt y_{\ell-1}}{y_{\ell-1}^T p_{\ell-1}}$ \\ \hline
\end{tabular}
\end{table}

\subsection{{\tt JFS3dir1}-based methods}

\begin{table}[htbp]
\centering
\caption{{\tt JFS3dir1}-based methods.}
\label{tab:jfs3dir1}
\renewcommand{\arraystretch}{1.3} 
\begin{tabular}{l c c}
\hline
\textbf{Method \& Direction} & $\boldsymbol{\beta^1_\ell}$ & $\boldsymbol{\beta^2_\ell}$ \\ \hline
{\tt PLS-S-Z1} and {\tt PF-S-Z1} / {\tt JFS3dir1-Z1} & 
$\displaystyle \frac{y_{\ell-1}^T E(x_\ell)}{\|E(x_{\ell-1})\|^2}$ & 
$\displaystyle -\frac{p_{\ell-1}^T E(x_\ell)}{\|E(x_{\ell-1})\|^2}$ \\[2ex]

{\tt PLS-S-Z2} and {\tt PF-S-Z2} / {\tt JFS3dir1-Z2} & 
$\displaystyle \frac{y_{\ell-1}^T E(x_\ell)}{y_{\ell-1}^T p_{\ell-1}}$ & 
$\displaystyle -\frac{p_{\ell-1}^T E(x_\ell)}{y_{\ell-1}^T p_{\ell-1}}$ \\[1ex] \hline
\end{tabular}
\end{table}

\subsection{{\tt JFS3dir2}-based methods}

Define $\phi_{\ell-1}
=
\D\frac{\|y_{\ell-1}\|^2}{y_{\ell-1}^Ts_{\ell-1}}$, 
$\phi_{\ell-1}^{1,j} = 1 + j\phi_{\ell-1}$ for $j=1,2$, 
$\phi_{\ell-1}^2 = 1 - \min\{1,\phi_{\ell-1}\}$.

\begin{table}[htbp]
\centering
\caption{{\tt JFS3dir2}-based methods.}
\label{tab:jfs3dir2}
\renewcommand{\arraystretch}{1.3} 
\begin{tabular}{l c c}
\hline
\textbf{Method \& Direction} & $\boldsymbol{\beta^3_\ell}$ & $\boldsymbol{\beta^4_\ell}$ \\ \hline

{\tt PLS-S-An1} and {\tt PF-S-An1} / {\tt JFS3dir2-An1} & 
$\displaystyle -\frac{s_{\ell-1}^T E(x_\ell)}{y_{\ell-1}^T s_{\ell-1}}$ & 
$\displaystyle \phi_{\ell-1}^{1,1} \beta^3_\ell + \frac{y_{\ell-1}^T E(x_\ell)}{y_{\ell-1}^T s_{\ell-1}}$ \\[2.5ex]

{\tt PLS-S-An2} and {\tt PF-S-An2} / {\tt JFS3dir2-An2} & 
$\displaystyle -\frac{s_{\ell-1}^T E(x_\ell)}{y_{\ell-1}^T s_{\ell-1}}$ & 
$\displaystyle \phi_{\ell-1}^{1,2} \beta^3_\ell + \frac{y_{\ell-1}^T E(x_\ell)}{y_{\ell-1}^T s_{\ell-1}}$ \\[2.5ex]

{\tt PLS-S-De} and {\tt PF-S-De} / {\tt JFS3dir2-De} & 
$\displaystyle -\frac{s_{\ell-1}^T E(x_\ell)}{y_{\ell-1}^T s_{\ell-1}}$ & 
$\displaystyle \phi_{\ell-1}^2 \beta^3_\ell + \frac{y_{\ell-1}^T E(x_\ell)}{y_{\ell-1}^T s_{\ell-1}}$ \\[1.5ex] \hline
\end{tabular}
\end{table}

\subsection{{\tt JFS3dir3}-based methods}

For the {\tt JFS3dir3} family, let $\overline{\gamma}>0$, ${\tt Re}
=
\frac{
E(x_\ell)^T p_{\ell-1}
}{
\|E(x_{\ell-1})\|\,\|p_{\ell-1}\|
}$, 
and let
\[
\beta^5_{\ell}
=
-\frac{
(\gamma_{\ell}-1)\|E(x_\ell)\|^2
+
\beta_{\ell}E(x_\ell)^T p_{\ell-1}
}{
E(x_\ell)^T h_\ell
}.
\]
Here, $\gamma_{\ell}
:=
\max\bigl\{
\varsigma_1,
\min\{\varsigma_2,\bar\gamma_{\ell}\}
\bigr\}$ and 
$0<\varsigma_1\le1<\varsigma_2$.

\begin{table}[htbp]
\centering
\caption{{\tt JFS3dir3}-based methods.}
\label{tab:jfs3dir3}
\renewcommand{\arraystretch}{1.2} 
\begin{tabular}{l c}
\hline
\textbf{Method \& Direction} & $\boldsymbol{\bar\gamma_\ell}$ \\ \hline
{\tt PLS-S-A1} and {\tt PF-S-A1} / {\tt JFS3dir3-A1} & $1-\overline{\gamma}|\beta_\ell{\tt Re}|$ \\[0.5ex]
{\tt PLS-S-A2} and {\tt PF-S-A2} / {\tt JFS3dir3-A2} & $1+\overline{\gamma}|\beta_\ell{\tt Re}|$ \\[0.5ex]
{\tt PLS-S-A3} and {\tt PF-S-A3} / {\tt JFS3dir3-A3} & $1-\overline{\gamma}\beta_\ell{\tt Re}$ \\[0.5ex]
{\tt PLS-S-A4} and {\tt PF-S-A4} / {\tt JFS3dir3-A4} & $1+\overline{\gamma}\beta_\ell{\tt Re}$ \\[0.5ex]
{\tt PLS-S-A5} and {\tt PF-S-A5} / {\tt JFS3dir3-A5} & $1-\overline{\gamma}|{\tt Re}|$ \\[0.5ex]
{\tt PLS-S-A6} and {\tt PF-S-A6} / {\tt JFS3dir3-A6} & $1+\overline{\gamma}|{\tt Re}|$ \\[0.5ex]
{\tt PLS-S-A7} and {\tt PF-S-A7} / {\tt JFS3dir3-A7} & $1-\overline{\gamma}{\tt Re}$ \\[0.5ex]
{\tt PLS-S-A8} and {\tt PF-S-A8} / {\tt JFS3dir3-A8} & $1+\overline{\gamma}{\tt Re}$ \\ \hline
\end{tabular}
\end{table}

\subsection{Appendix: Well-definedness of subspace parameters}
\label{app:subparam}

Before addressing the well-definedness of the subspace updating parameters, we
briefly describe the structure of the coefficients $\beta_\ell$ and
$\beta_\ell^i$, $i=1,\ldots,5$, summarized in
Tables~\ref{tab:jfs2}--\ref{tab:jfs3dir3}. Although these coefficients appear
in several algebraic forms, they are constructed from a small collection of
quantities associated with two successive iterations, namely the residuals
$E(x_\ell)$ and $E(x_{\ell-1})$, the displacement $s_{\ell-1}$, the previous
direction $p_{\ell-1}$, and the residual difference $y_{\ell-1}$. Consequently,
several coefficients contain common denominators, such as
$\|E(x_{\ell-1})\|^2$, $y_{\ell-1}^Tp_{\ell-1}$, and
$y_{\ell-1}^Ts_{\ell-1}$.

The following result shows that these coefficients are well defined at every
iteration at which the algorithm has not terminated, provided that the
specified safeguarding rule is applied whenever a curvature denominator
vanishes.

\begin{prop}\label{pr3}
Suppose that iteration $\ell-1$ is nonstationary, so that $\|E(x_{\ell-1})\|>\eps$. Assume that every occurrence of $y_{\ell-1}$ in a denominator is replaced, whenever necessary, by
the safeguarded vector defined below so that $s_{\ell-1}^Ty_{\ell-1}\neq0$ and $y_{\ell-1}^Tp_{\ell-1}\neq0$, and that the auxiliary vector $h_\ell$ in $\beta_{\ell}^5$ is selected so that every denominator
involving $h_\ell$ is nonzero. Then all subspace coefficients
$\beta_\ell$ and $\beta_\ell^i$, $i=1,\ldots,5$, appearing in
Tables~\ref{tab:jfs2}--\ref{tab:jfs3dir3} are well-defined.
\end{prop}

\begin{proof}
If $\|E(x_{\ell-1})\|
\le
\eps$, then the stopping criterion is satisfied and no new subspace coefficient is
required. We therefore consider an iteration for which $\|E(x_{\ell-1})\|
>
\eps$. In particular, $\|E(x_{\ell-1})\|^2
>0$, so every denominator containing $\|E(x_{\ell-1})\|^2$ is strictly positive.

Since the iteration is nonstationary, the fallback and scaling rules imply $p_{\ell-1}\neq0$ and $p_{\ell-1}^{\scal}\neq0$. By \eqref{e.xbar}, $s_{\ell-1}
=
\bar x_{\ell-1}-x_{\ell-1}
=
\mu_{\ell-1}p_{\ell-1}^{\scal}$. For the line-search step-size implementation,
Lemma~\ref{le:lowerstep} gives
$\mu_{\ell-1}
\ge
\mu_{\min}
>
0$. For the fixed step-size implementation, $\mu_{\ell-1}
=
\alpha
>
0$. Hence, $s_{\ell-1}\neq0$.
If $s_{\ell-1}^Ty_{\ell-1}\neq0$, set $y_{\ell-1}^{(1)}:=y_{\ell-1}$. Otherwise, define
$y_{\ell-1}^{(1)}
:=
y_{\ell-1}
+
\varphi_{\ell-1}s_{\ell-1}$, where $0<\varphi_{\ell-1}\le\overline\varphi$ for some constant $\overline\varphi>0$. Since $s_{\ell-1}\neq0$,
\[
s_{\ell-1}^Ty_{\ell-1}^{(1)}
=
s_{\ell-1}^Ty_{\ell-1}
+
\varphi_{\ell-1}\|s_{\ell-1}\|^2
=
\varphi_{\ell-1}\|s_{\ell-1}\|^2
>
0.
\]
If $(y_{\ell-1}^{(1)})^Tp_{\ell-1}\neq0$, set $\widetilde y_{\ell-1}:=y_{\ell-1}^{(1)}$. Otherwise, define
$u_{\ell-1}
:=
p_{\ell-1}
-
\frac{s_{\ell-1}^Tp_{\ell-1}}{\|s_{\ell-1}\|^2}s_{\ell-1}$.
If $u_{\ell-1}=0$, then $p_{\ell-1}$ is a nonzero multiple of $s_{\ell-1}$, which together with $s_{\ell-1}^Ty_{\ell-1}^{(1)}\neq0$ would imply $(y_{\ell-1}^{(1)})^Tp_{\ell-1}\neq0$, a contradiction. Hence, $u_{\ell-1}\neq0$. Define
$\widetilde y_{\ell-1}
:=
y_{\ell-1}^{(1)}
+
\psi_{\ell-1}\|s_{\ell-1}\|
u_{\ell-1}/\|u_{\ell-1}\|$, where $0<\psi_{\ell-1}\le\overline\psi$ for some constant $\overline\psi>0$. Since $s_{\ell-1}^Tu_{\ell-1}=0$,
\[
s_{\ell-1}^T\widetilde y_{\ell-1}
=
s_{\ell-1}^Ty_{\ell-1}^{(1)}
\neq0.
\]
Moreover, since $(y_{\ell-1}^{(1)})^Tp_{\ell-1}=0$ and $p_{\ell-1}^Tu_{\ell-1}=\|u_{\ell-1}\|^2$,
\[
\widetilde y_{\ell-1}^Tp_{\ell-1}
=
\psi_{\ell-1}\|s_{\ell-1}\|\|u_{\ell-1}\|
>
0.
\]
Therefore, after replacing $y_{\ell-1}$ by
$\widetilde y_{\ell-1}$, all denominators involving
$y_{\ell-1}^Ts_{\ell-1}$ or $y_{\ell-1}^Tp_{\ell-1}$ are nonzero.

The remaining denominators in the definitions of
$\beta_\ell$ and $\beta_\ell^i$, $i=1,\ldots,4$, are therefore nonzero.
For $\beta_\ell^5$, the auxiliary vector $h_\ell$ is selected so that the
denominator appearing in its definition is nonzero. In particular, whenever
that denominator contains $E(x_\ell)^Th_\ell$, we require $E(x_\ell)^Th_\ell
\neq
0$. Thus, all coefficients $\beta_\ell$ and $\beta_\ell^i$,
$i=1,\ldots,5$, are well-defined. For notational simplicity, after safeguarding we replace
$y_{\ell-1}$ by $\widetilde y_{\ell-1}$ and continue to denote the resulting
vector by $y_{\ell-1}$. \qed
\end{proof}

\begin{rem}
The safeguarding modification is used only when a curvature denominator
vanishes. It preserves the Jacobian-free structure of the method and ensures
that the corresponding subspace coefficients are well defined. Because the
modified vector may enter the construction of the next search direction, the
safeguard can change subsequent trial points and iterates; it should therefore
be regarded as part of the direction-generation rule rather than as a purely
notational modification. Moreover,
$\|\widetilde y_{\ell-1}-y_{\ell-1}\|
\le
(\overline\varphi+\overline\psi)\|s_{\ell-1}\|$. Since $\|s_{\ell-1}\|=
\|\bar x_{\ell-1}-x_{\ell-1}\|
\to
0$ by Lemma~\ref{le:bo}, it follows that $\|\widetilde y_{\ell-1}-y_{\ell-1}\|
\to
0$. Hence, the safeguarding perturbation vanishes asymptotically. The global
convergence and rate analyses remain applicable provided that the safeguarded
directions continue to satisfy the scaling and angle conditions imposed by
Algorithm \ref{a.FDS}.
\end{rem}
\end{sloppypar}

\begin{thebibliography}{10}

\bibitem{aho}
M.~Ahookhosh, K.~Amini, and S.~Bahrami.
\newblock Two derivative-free projection approaches for systems of large-scale
  nonlinear monotone equations.
\newblock {\em Numer. Algorithms\/} {\bf 64} (2013), 21--42.

\bibitem{al2}
M.~Al-Baali, A.~Caliciotti, G.~Fasano, and M.~Roma.
\newblock Exploiting damped techniques for nonlinear conjugate gradient
  methods.
\newblock {\em Math. Methods Oper. Res.\/} {\bf 86} (2017), 501--522.

\bibitem{al1}
M.~Al-Baali, Y.~Narushima, and H.~Yabe.
\newblock A family of three-term conjugate gradient methods with sufficient
  descent property for unconstrained optimization.
\newblock {\em Comput. Optim. Appl.\/} {\bf 60} (2015), 89--110.

\bibitem{Andrei22013}
N.~Andrei.
\newblock On three-term conjugate gradient algorithms for unconstrained
  optimization.
\newblock {\em Appl. Math. Comput.\/} {\bf 219} (February 2013), 6316--6327.

\bibitem{Andrei12013}
N.~Andrei.
\newblock A simple three-term conjugate gradient algorithm for unconstrained
  optimization.
\newblock {\em J. Comput. Appl. Math.\/} {\bf 241} (March 2013), 19--29.

\bibitem{antipin1997equilibrium}
A.~S Antipin.
\newblock Equilibrium programming: proximal methods.
\newblock {\em Comput. Math. Math. Phys\/} {\bf 37} (1997), 1285--1296.

\bibitem{audet2006mads}
C.~Audet and J.~E. Dennis.
\newblock Mesh adaptive direct search algorithms for constrained optimization.
\newblock {\em SIAM Journal on Optimization\/} {\bf 17} (2006), 188--217.

\bibitem{bab1}
S.~Babaie-Kafaki.
\newblock Two modified scaled nonlinear conjugate gradient methods.
\newblock {\em J. Comput. Appl. Math\/} {\bf 261} (2014), 172--182.

\bibitem{bab2}
S.~Babaie-Kafaki and R.~Ghanbari.
\newblock Two optimal {D}ai--{L}iao conjugate gradient methods.
\newblock {\em Optimization\/} {\bf 64} (2015), 2277--2287.

\bibitem{Bauschke2017}
H.~H. Bauschke and P.~L. Combettes.
\newblock {\em Convex Analysis and Monotone Operator Theory in Hilbert Spaces}.
\newblock Springer International Publishing  (2017).

\bibitem{bohm}
A.~Bohm, M.~Sedlmayer, E.~R. Csetnek, and R.~I. Bot.
\newblock Two steps at a time---taking gan training in stride with tseng's
  method.
\newblock {\em SIAM Journal on Mathematics of Data Science\/} {\bf 4} (2022),
  750--771.

\bibitem{Bo2023}
R.~I. Bo{\c{t}}, E.~R. Csetnek, and D.~K. Nguyen.
\newblock Fast optimistic gradient descent ascent ({OGDA}) method in continuous
  and discrete time.
\newblock {\em Foundations of Computational Mathematics\/} {\bf 25} (November
  2023), 163--222.

\bibitem{BS}
A.~Bouaricha and R.~B. Schnabel.
\newblock Tensor methods for large sparse systems of nonlinear equations.
\newblock {\em Math. Program.\/} {\bf 82} (August 1998), 377--400.

\bibitem{BKL}
S.~Buhmiler, N.~Kreji{\'{c}}, and Z.~Lu{\v{z}}anin.
\newblock Practical quasi-newton algorithms for singular nonlinear systems.
\newblock {\em Numer. Algorithms\/} {\bf 55} (February 2010), 481--502.

\bibitem{chavdarova2021last}
T.~Chavdarova, M.~I. Jordan, and M.~Zampetakis.
\newblock Last-iterate convergence of saddle point optimizers via
  high-resolution differential equations.
\newblock {\em arXiv preprint arXiv:2112.13826\/} (2021).

\bibitem{conn2009dfo}
A.~R. Conn, K.~Scheinberg, and L.~N. Vicente.
\newblock {\em Introduction to Derivative-Free Optimization}.
\newblock SIAM  (2009).

\bibitem{dl}
H.~Y. Dai and L.~Z. Liao.
\newblock New conjugacy conditions and related nonlinear conjugate gradient
  methods.
\newblock {\em App Math Optim\/} {\bf 43} (2001), 87--101.

\bibitem{dai99}
Y.~H. Dai and Y.~Yuan.
\newblock A nonlinear conjugate gradient method with a strong global
  convergence property.
\newblock {\em SIAM J. Optim.\/} {\bf 10} (1999), 177--182.

\bibitem{dai}
Z.~Dai and H.~Zhu.
\newblock A modified {H}estenes-{S}tiefel-type derivative-free method for
  large-scale nonlinear monotone equations.
\newblock {\em Mathematics\/} {\bf 8} (2020), 168.

\bibitem{Deng2015}
S.~Deng and Z.~Wan.
\newblock A three-term conjugate gradient algorithm for large-scale
  unconstrained optimization problems.
\newblock {\em Appl Numer Math\/} {\bf 92} (June 2015), 70--81.

\bibitem{deuflhard2011newton}
P.~Deuflhard.
\newblock {\em Newton methods for nonlinear problems: affine invariance and
  adaptive algorithms}, Vol.~35.
\newblock Springer Science \& Business Media  (2011).

\bibitem{DolM}
E.~D. Dolan and J.~J. Mor{\'{e}}.
\newblock Benchmarking optimization software with performance profiles.
\newblock {\em Math. Program.\/} {\bf 91} (January 2002), 201--213.

\bibitem{flre}
R.~Fletcher and C.~M. Reeves.
\newblock Function minimization by conjugate gradients.
\newblock {\em The computer journal\/} {\bf 7} (1964), 149--154.

\bibitem{goodfellow2020generative}
I.~Goodfellow, J.~Pouget-Abadie, M.~Mirza, B.~Xu, D.~Warde-Farley, S.~Ozair,
  A.~Courville, and Y.~Bengio.
\newblock Generative adversarial networks.
\newblock {\em Communications of the ACM\/} {\bf 63} (2020), 139--144.

\bibitem{gorbunov2022stochastic}
E.~Gorbunov, H.~Berard, G.~Gidel, and N.~Loizou.
\newblock Stochastic extragradient: General analysis and improved rates.
\newblock In {\em International Conference on Artificial Intelligence and
  Statistics}, pp. 7865--7901. {PMLR}  (2022).

\bibitem{griewank1981modification}
A.~Griewank.
\newblock The modification of {N}ewton’s method for unconstrained
  optimization by bounding cubic terms.
\newblock Technical report, Technical report NA/12  (1981).

\bibitem{HZ0}
W.~W. Hager and H.~Zhang.
\newblock A survey of nonlinear conjugate gradient methods.
\newblock {\em Pac. J. Optim.\/} {\bf 2} (2006), 35--58.

\bibitem{halpern1967fixed}
B.~Halpern.
\newblock Fixed points of nonexpanding maps.
\newblock {\em Bulletin of the American Mathematical Society\/} {\bf 73}
  (1967), 957--961.

\bibitem{hes}
M.~R. Hestenes, E.~Stiefel, et~al.
\newblock {\em Methods of conjugate gradients for solving linear systems},
  Vol.~49.
\newblock NBS Washington, DC  (1952).

\bibitem{HP}
A.~H. Ibrahim, P.~Kumam, A.~B. Abubakar, W.~Jirakitpuwapat, and J.~Abubakar.
\newblock A hybrid conjugate gradient algorithm for constrained monotone
  equations with application in compressive sensing.
\newblock {\em Heliyon\/} {\bf 6} (March 2020), e03466.

\bibitem{NW}
S.~Wright J.~Nocedal.
\newblock {\em Numerical Optimization}.
\newblock Springer New York  (1999).

\bibitem{jiao2024minimax}
Y.~Jiao and G.~Li.
\newblock Minimax-optimal multi-agent robust reinforcement learning.
\newblock {\em arXiv preprint arXiv:2412.19873\/} (2024).

\bibitem{Jin2021}
C.~Jin, P.~Netrapalli, R.~Ge, S.~M. Kakade, and M.~I. Jordan.
\newblock On nonconvex optimization for machine learning: Gradients,
  stochasticity, and saddle points.
\newblock {\em Journal of the ACM\/} {\bf 68} (February 2021), 1--29.

\bibitem{KYF}
C.~Kanzow, N.~Yamashita, and M.~Fukushima.
\newblock Levenberg{\textendash}{M}arquardt methods with strong local
  convergence properties for solving nonlinear equations with convex
  constraints.
\newblock {\em J. Comput. Appl. Math.\/} {\bf 172} (December 2004), 375--397.

\bibitem{kimiaei2017new}
M.~Kimiaei.
\newblock A new class of nonmonotone adaptive trust-region methods for
  nonlinear equations with box constraints.
\newblock {\em Calcolo\/} {\bf 54} (2017), 769--812.

\bibitem{VRDFON}
M.~Kimiaei.
\newblock An improved randomized algorithm with noise level tuning for
  large-scale noisy unconstrained {DFO} problems.
\newblock {\em Numerical Algorithms\/} (January 2025).

\bibitem{LMLS}
M.~Kimiaei and A.~Neumaier.
\newblock A new limited memory method for unconstrained nonlinear least
  squares.
\newblock {\em Soft Computing\/} {\bf 26} (December 2021), 465–490.

\bibitem{VRBBO}
M.~Kimiaei and A.~Neumaier.
\newblock Efficient unconstrained black box optimization.
\newblock {\em {M}ath. {P}rogram. {C}omput.\/} (2022), 365--414.

\bibitem{MADFO}
M.~Kimiaei and A.~Neumaier.
\newblock Effective matrix adaptation strategy for noisy derivative-free
  optimization.
\newblock {\em Mathematical Programming Computation\/} {\bf 16} (July 2024),
  459--501.

\bibitem{MATRS}
M.~Kimiaei and A.~Neumaier.
\newblock {MATRS}: heuristic methods for noisy derivative-free
  bound-constrained mixed-integer optimization.
\newblock {\em Mathematical Programming Computation\/} {\bf 17} (May 2025),
  505–546.

\bibitem{SSDFO}
M.~Kimiaei, A.~Neumaier, and P.~Faramarzi.
\newblock New subspace method for unconstrained derivative-free optimization.
\newblock {\em ACM Transactions on Mathematical Software\/} {\bf 49} (December
  2023), 1--28.

\bibitem{suppMat}
M.~Kimiaei, S.~Shabani, and M.~Breu{\ss}.
\newblock Supplementary {Material} and {MATLAB} code: Subspace methods for
  min--max problems.
\newblock \url{https://github.com/GS1400/PLS-PF-S}  (2026).

\bibitem{korpelevich1976extragradient}
G.~M. Korpelevich.
\newblock The extragradient method for finding saddle points and other
  problems.
\newblock {\em Matecon\/} {\bf 12} (1976), 747--756.

\bibitem{LF}
D.~Li and M.~Fukushima.
\newblock A globally and superlinearly convergent {G}auss--{N}ewton-based
  {BFGS} method for symmetric nonlinear equations.
\newblock {\em SIAM J. Math. Anal.\/} {\bf 37} (January 1999), 152--172.

\bibitem{liu1}
Y.~Liu and C.~Storey.
\newblock Efficient generalized conjugate gradient algorithms, part 1: theory.
\newblock {\em J Optim Theory Appl\/} {\bf 69} (1991), 129--137.

\bibitem{SDBOX}
S.~Lucidi and M.~Sciandrone.
\newblock A derivative-free algorithm for bound constrained optimization.
\newblock {\em Computational Optimization and Applications\/} {\bf 21} (2002),
  119--142.

\bibitem{madr}
A.~Madry, A.~Makelov, L.~Schmidt, D.~Tsipras, and A.~Vladu.
\newblock Towards deep learning models resistant to adversarial attacks.
\newblock  (2018).

\bibitem{Malitsky2019}
Y.~Malitsky.
\newblock Golden ratio algorithms for variational inequalities.
\newblock {\em Mathematical Programming\/} {\bf 184} (2020), 383--410.

\bibitem{NA}
A.~Neumaier and B.~Azmi.
\newblock Line search and convergence in bound-constrained optimization.
\newblock {\em Optimization Online\/} (2019).

\bibitem{CLS}
Arnold Neumaier and Morteza Kimiaei.
\newblock An improvement of the goldstein line search.
\newblock {\em Optimization Letters\/} {\bf 18} (April 2024), 1313--1333.

\bibitem{nocedal2006numerical}
J.~Nocedal and S.~J. Wright.
\newblock {\em Numerical optimization}.
\newblock Springer  (2006).

\bibitem{omid}
S.~Omidshafiei, J.~Pazis, C.~Amato, J.~P. How, and J.~Vian.
\newblock Deep decentralized multi-task multi-agent reinforcement learning
  under partial observability.
\newblock In {\em International conference on machine learning}, pp.
  2681--2690. PMLR  (2017).

\bibitem{Ou2022}
Y.~Ou and L.~Li.
\newblock A unified convergence analysis of the derivative-free
  projection-based method for constrained nonlinear monotone equations.
\newblock {\em Numerical Algorithms\/} {\bf 93} (December 2022), 1639--1660.

\bibitem{pol}
E.~Polak and G.~Ribiere.
\newblock Note sur la convergence de m{\'e}thodes de directions conjugu{\'e}es.
\newblock {\em ESAIM: Mathematical Modelling and Numerical
  Analysis-Mod{\'e}lisation Math{\'e}matique et Analyse Num{\'e}rique\/} {\bf
  3} (1969), 35--43.

\bibitem{popov1980modification}
L.~D. Popov.
\newblock A modification of the {A}rrow-{H}urwicz method for search of saddle
  points.
\newblock {\em Mathematical notes of the Academy of Sciences of the USSR\/}
  {\bf 28} (1980), 845--848.

\bibitem{UOBYQA}
M.~J.~D. Powell.
\newblock {UOBYQA}: Unconstrained optimization by quadratic approximation.
\newblock {\em Mathematical Programming\/} {\bf 92} (2002), 555--582.

\bibitem{BOBYQA}
M.~J.~D. Powell.
\newblock The {BOBYQA} algorithm for bound constrained optimization without
  derivatives.
\newblock Technical report, Department of Applied Mathematics and Theoretical
  Physics, Cambridge  (2009).

\bibitem{robey2021adversarial}
A.~Robey, L.~Chamon, G.~J. Pappas, H.~Hassani, and A.~Ribeiro.
\newblock Adversarial robustness with semi-infinite constrained learning.
\newblock {\em Advances in Neural Information Processing Systems\/} {\bf 34}
  (2021), 6198--6215.

\bibitem{ryu2016primer}
E.~K. Ryu and S.~Boyd.
\newblock Primer on monotone operator methods.
\newblock {\em Appl. comput. math\/} {\bf 15} (2016), 3--43.

\bibitem{SS}
M.~V. Solodov and B.~F. Svaiter.
\newblock A globally convergent inexact {N}ewton method for systems of monotone
  equations.
\newblock In {\em Reformulation: Nonsmooth, Piecewise Smooth, Semismooth and
  Smoothing Methods}, pp. 355--369. Springer {US}  (1998).

\bibitem{sorber2012unconstrained}
L.~Sorber, M.~V. Barel, and L.~D. Lathauwer.
\newblock Unconstrained optimization of real functions in complex variables.
\newblock {\em SIAM Journal on Optimization\/} {\bf 22} (2012), 879--898.

\bibitem{WW}
A.~Wood and B.~Wollenberg.
\newblock {\em Power Generation Operation and Control}.
\newblock John Wiley and Sons  (1996).

\bibitem{XZ}
Y.~Xiao and H.~Zhu.
\newblock A conjugate gradient method to solve convex constrained monotone
  equations with applications in compressive sensing.
\newblock {\em J. Math. Anal. Appl\/} {\bf 405} (September 2013), 310--319.

\bibitem{yoon2021accelerated}
T.~Yoon and E.~K. Ryu.
\newblock Accelerated algorithms for smooth convex-concave minimax problems
  with o (1/k\^{} 2) rate on squared gradient norm.
\newblock In {\em International Conference on Machine Learning}, pp.
  12098--12109. PMLR  (2021).

\bibitem{ZZL1}
L.~Zhang, W.~Zhou, and D.~H. Li.
\newblock A descent modified
  {P}olak{\textendash}{R}ibi{\`{e}}re{\textendash}{P}olyak conjugate gradient
  method and its global convergence.
\newblock {\em IMA J. Numer. Anal.\/} {\bf 26} (October 2006), 629--640.

\bibitem{ZZL2}
L.~Zhang, W.~Zhou, and D.~H. Li.
\newblock Some descent three-term conjugate gradient methods and their global
  convergence.
\newblock {\em Optim Methods Softw\/} {\bf 22} (August 2007), 697--711.

\end{thebibliography}
\end {document}